\documentclass[12pt, reqno]{amsart}
\usepackage{amsmath, amsthm, amscd, amsfonts, amssymb, graphicx, color}
\usepackage[bookmarksnumbered, colorlinks, plainpages]{hyperref}
\hypersetup{colorlinks=true,linkcolor=red, anchorcolor=green, citecolor=cyan, urlcolor=red, filecolor=magenta, pdftoolbar=true}

\newtheorem{theorem}{Theorem}[section]
\newtheorem{lemma}[theorem]{Lemma}
\newtheorem{proposition}[theorem]{Proposition}

\theoremstyle{definition}
\newtheorem{definition}[theorem]{Definition}
\newtheorem{example}[theorem]{Example}

\theoremstyle{remark}
\newtheorem{remark}[theorem]{Remark}
\numberwithin{equation}{section}

\begin{document}
\setcounter{page}{1}

\author[Bhat, B. V. Rajarama]{B. V. Rajarama Bhat}
\author[Mukherjee, M.]{Mithun  Mukherjee}
\title[TWO STATES]{ TWO STATES}
\address[Bhat, B. V. Rajarama]{Theoretical Statistics and Mathematics Unit, Indian Statistical Institute, R. V. College Post, Bangalore-560059, India}
\email{bhat@isibang.ac.in}
\address[Mukherjee, M.]
	{School of Mathematics, IISER, Thiruvananthapuram, Maruthamala PO, Vithura, Kerala - 695551, India} 

\address{Current Address: School of Mathematical \& Computational Sciences, Indian Association for the Cultivation of Science, Jadavpur, Kolkata, 700032, India}
\email{mithun.mukherjee@iacs.res.in}

\maketitle
\vskip 10pt

\begin{center}
{\bf Abstract\footnote {AMS Subject Classification: 46L30, 46L08.
Keywords: States, Completely positive maps, Hilbert $C^*$-modules,
Bures distance.}}

\end{center}

\vskip 4pt

\begin{abstract}

D. Bures defined a metric  $\beta $ on states of a $C^*$-algebra and
this concept has been generalized to unital completely positive maps
$\phi : \mathcal A \to \mathcal B$, where $\mathcal B$ is either an
injective $C^*$-algebra or a von Neumann algebra. We introduce a new
distance $\gamma $ for the same classes of unital completely
positive maps. We use in our definition the distance between
representations on the same Hilbert $C^*$-module in contrast to the
Bures metric which uses one representation and distinct vectors.
This metric can be expressed in terms of a class of  completely
positive maps on free products of $C^*$-algebras and in this setting
$\gamma $ looks like Wasserstein metric on probability measures.
Surprisingly, when the range algebra $\mathcal B$ is injective,
$\gamma $ and $\beta $ are related by the following explicit
formula: $\beta ^2= 2-\sqrt{4- \gamma ^2} .$ A deep result of Choi
and Li on constrained dilation is the main tool in proving this
formula.


\end{abstract}

\maketitle

\section{Introduction}

Given a state $\phi$ on a unital $C^*$-algebra $\mathcal A,$ the
well-known Gelfand-Naimark-Segal (GNS)-construction yields a triple
$(H, \pi, x),$ where $H$ is a Hilbert space, $\pi : \mathcal A
\rightarrow \mathbb B(H)$ is a representation ($*$-homomorphism) on
the algebra $\mathbb B(H)$ of all bounded operators on $H$, and
$x\in H$ is a vector such that $\phi(\cdot) = \langle x, \pi(\cdot)x
\rangle.$ What is the geometry of states from the point of view of
their GNS representations is a natural question. More specifically,
given two states which are close in norm,  can we choose GNS triples
for them, which are also close in some sense?  It is tricky to
measure closeness of two triples. Bures \cite{Bur} took the
following approach.

Bures distance  of two states $\phi_1,\phi_2$ on $\mathcal B,$ is
defined as
  $\beta(\phi_1,\phi _2)=\mbox{inf}\|x_1-x_2\|$, where the infimum is taken over all GNS-triples
  with `common' representation spaces: $(H, \pi, x_1), (H, \pi, x_2)$ of $\phi_1, \phi_2.$ Here in
  two GNS triples, two of the components namely the Hilbert space and the representation are taken
  to be common, and the distance is measured only for the vectors. Perhaps, it would be equally
  natural to define another notion $\gamma $ as, $\gamma (\phi _1, \phi _2)= \mbox{inf}
   \| \pi _1-\pi _2\| _{cb}$, where $\| \cdot \| _{cb}$ stands for completely bounded norm and the
   infimum is now taken over all GNS-triples with `joint' representation spaces: $(H, \pi _1, x),
   (H, \pi _2, x)$ of $\phi _1, \phi _2.$ In other words, the Hilbert space and vector are common and
    only the representations are different. We explore this notion here.

 Before we go further, we remark that this circle of ideas have natural extensions from states to
 completely positive maps and it is convenient to start with  such a general set up.  The notion of
 Bures metric for completely positive maps was introduced
at first for completely positive (CP) maps from a $C^*$-algebra $\mathcal A$ to $\mathbb B(G)$ for
some
Hilbert space $G$, by Kretschmann, Schlingemann and Werner in \cite{Con}.   In the same paper extension
 of the notion to more general range $C^*$-algebras has been presented through an alternative somewhat indirect
 definition of the Bures distance. The reason for this is that, these authors  use the Stinespring
  representation (\cite{Sti}) for the initial definition, which in the usual formulation requires
   the range space to be
the whole algebra $\mathbb B(G).$   This artificiality can be
removed if one uses the theory of Hilbert $C^*$-modules. This has been
carried out by Bhat and Sumesh \cite{BhS}. Making use of basic ideas
from Hilbert $C^*$-module theory, it is seen that $\beta$ is indeed
a  metric when the range algebra is an injective algebra or a von
Neumann algebra. A counter example
  is also presented in \cite{BhS} that one may not even get a metric when the range algebra
  is a general $C^*$-algebra.

The notion of Bures metric has found many mathematical and physical
applications ( \cite{Ara},  \cite{GMW}, \cite{Kos}). There has been
some renewed interest in the subject due to applications in quantum
information theory (\cite{SZ}, \cite{Uh1}, \cite{Uh2}). The
generalized version as distance for CP maps also has applications in
this field \cite{KSW}. Our interest in this topic stems from  its
usefulness in the study of generators of quantum dynamical
semigroups \cite{Muk}.

 Now the revised set up is as follows. Let $\phi _1, \phi _2$ be two completely positive maps
 from a unital $C^*$-algebra $\mathcal A$ to another unital $C^*$-algebra $\mathcal
 B.$ We call $\mathcal B$ as the range algebra.
 The Stinespring's theorem in Hilbert $C^*$-module language (see \cite{Pas}), provides {\em Stinespring
 triples\/} or GNS representation modules $(\mathcal E _1, \sigma _1, x_1), (\mathcal E_2, \sigma _2, x_2)$,
 where  $\mathcal E_i$ is a  Hilbert $\mathcal B$ module with a left action
  $\sigma _i$ of $\mathcal A$ on it and a vector $x_i\in \mathcal E_i$, such that
 $$\phi _i(\cdot )= \langle x_i, \sigma _i(\cdot )x_i\rangle ,$$
 for $i=1,2$.  Then mimicking the definitions for states we
 have $$\beta (\phi _1 , \phi _2)= \inf \|x_1-x_2\|,$$ where the infimum is taken over all
 `common' representation modules $(\mathcal E, \sigma , x_1),$ $ (\mathcal E, \sigma , x_2)$ of
 $ (\phi _1, \phi _2).$ Similarly, we can define,
 $$\gamma (\phi _1, \phi _2)= \inf \| \sigma _1-\sigma _2\| _{cb},$$ where `$cb$' stands for the completely bounded norm and
 the infimum is taken
 over all `joint' representation modules $(\mathcal E, \sigma _1, x),$  $ (\mathcal E, \sigma _2, x)$ of
  $(\phi _1, \phi _2).$

 For lack of better name, we call  $\gamma $ as
`representation metric'. In this paper we study basic
 properties of    $\gamma (\phi _1, \phi _2)$ and its relationship with
 $\beta (\phi _1, \phi _2)$. We restrict ourselves to unital completely positive maps.
   We show that $\gamma $ is indeed a metric if the range algebra under
consideration is a von Neumann algebra or an injective
$C^*$-algebra,  exactly like in the case of $\beta $. In \cite{BhS}
an example can be seen where the triangle inequality fails for
$\beta $ (Of course, then the range algebra is a non-injective
$C^*$-algebra). We do not yet have any such  example for $\gamma .$

   The organisation of the paper is as follows. In Section 2, we present
   some basics on Hilbert $C^*$-modules and von    Neumann modules. We describe the GNS module associated with
   a completely positive map.

   In Section 3, we define $\gamma $. We show that $\gamma $ is a metric
    when the range algebra under consideration is a von Neumann algebra. We also show that the representation metric $\gamma $
     does not change on ampliation.

   In Section 4,  we show that the new notion $\gamma $ has a beautiful description in terms of
    the full free product of $C^*$-algebras. This is a feature not seen for Bures metric. This association
    with free product allows us to interpret representation metric as a notion coming from `joint distributions with
    given marginals', comparable to Wasserstein metric of probability measures. In a broad sense it is also somewhat like
    Gromov-Hausdorff distance for metric spaces. This allows us to
    show that $\gamma $ is a metric when the range algebra is an
    injective $C^*$-algebra.

      In Section 5, we address the attainability issue of the representation metric, i.e., we show that
      given two unital completely positive maps, there is a joint representing module in which the
   representation metric attains its value. This holds true when the
   range algebra is a von Neumann algebra or is an injective
   $C^*$-algebra.

     In Section 6, we establish a very interesting
direct relation of this metric with Bures metric. This  may be
considered as a new formula to compute Bures metric. We prove the
result for states and then extend it to the  case of  injective
$C^*$-algebras as range algebras. We do not know whether this is
true for von Neumann algebras. It is also unclear as to what is the
suitable extension of this theory to non-unital CP maps.

Finally, in the last section we have several examples and counter examples. In particular we show that the range algebra does matter for computing the
representation metric.

\section{Notation and basics of Hilbert $C^*$-modules}

Let $\mathcal B$ be a unital $C^*$-algebra. A complex vector space
$\mathcal E$ is a Hilbert  $C^*$ right $\mathcal B$-module (or
simply a Hilbert $\mathcal B$-module) if it is a right $\mathcal
B$-module with a $\mathcal B$-valued inner product, which is
complete with respect to the associated norm (see \cite{Lan},
\cite{Pas}, \cite{Ske} for basic theory). We denote the space of all
bounded and adjointable maps between two Hilbert $\mathcal
B$-modules $\mathcal E_1$ and $\mathcal E_2$ by
$\textbf{B}^a(\mathcal E _1,\mathcal E _2).$ In particular, if
$\mathcal E_1 = \mathcal E_2 = \mathcal E,$ then $=
\textbf{B}^a(\mathcal E):= \textbf{B}^a(\mathcal E,\mathcal E) ,$
and it is a unital $C^*$-algebra with natural algebraic operations
and operator norm.

 Suppose the $C^*$-algebra $\mathcal B$ is given concretely as an algebra of
 operators on a Hilbert space $G$, that is, $\mathcal B\subseteq
 \mathbb B(G).$ Here and elsewhere in this article whenever we have such an embedding
 without loss of
 generality  we would assume that the unit of $\mathcal B$ is same as
 the identity operator on $G.$ Now if $\mathcal E$ is a Hilbert
 $\mathcal B$-module, we would like to realize it as a subspace
 of  $\mathbb B(G,H)$ for some Hilbert space $H$, with
 $\langle Y, Z\rangle =Y^*Z$ in $\mathbb B(G,H).$ This is done as
 follows:  Given a Hilbert $\mathcal B$-module $\mathcal E,$
 we define the Hilbert
space $H := \mathcal E \odot G$ as the Hilbert space obtained from
the algebraic tensor product $\mathcal E \otimes G$, with semi-inner
product: $$\langle x \otimes g, x^\prime \otimes g^\prime \rangle :=
\langle g,
 (\langle x, x^\prime \rangle) g^\prime \rangle, ~~~\mbox{for}~~ x, x^\prime \in \mathcal E; g, g^\prime
  \in G;$$ after quotienting the space of null vectors, and completing.
  We denote the equivalence class containing $x \otimes g$ by $x \odot g.$ To each $x \in
  \mathcal E,$ we associate the
linear map $L_x : g \mapsto  x \odot g$ in $\mathbb B(G,H)$ with
adjoint $L^*_x : y \odot g \mapsto  (\langle x, y\rangle)g.$ Clearly
$L^*_xL_y = \langle x, y\rangle$ and $L_{xb} = L_x .(b)$ for all $x,
y \in \mathcal E,$ $b \in \mathcal B.$ Also $\|L_x\|^2 = \|(\langle
x, x\rangle)\| = \|x\|^2.$ By identifying $x$ with $L_x,$ we may
assume that $\mathcal E \subseteq \mathbb B(G,H).$ Note that $
a\mapsto  a\otimes id_G :
 \mathcal B^a(\mathcal E) \rightarrow \mathcal B(H)$ is a unital $*$-homomorphism and hence an
 isometry. So we may consider $\textbf B^a(\mathcal E) \subseteq \mathbb B(H).$

Suppose $\mathcal A$ is another unital $C^*$-algebra. A Hilbert
$\mathcal B$-module $\mathcal E$ is said to be a Hilbert $\mathcal
A-\mathcal B$ bi-module if there exists a unital $*$-representation
$\sigma : \mathcal A \rightarrow \textbf B^a(\mathcal E)$. If
$\mathcal E$ is a Hilbert $\mathcal A-\mathcal B$ bi-module, and the
left action $\sigma $ is fixed and there is no possibility of
confusion, we may describe the left action simply by $x\mapsto ax$,
 and thereby $\sigma(a)x = ax$ for all $x \in \mathcal E, a \in
\mathcal A.$  Further, if $\mathcal B\subseteq \mathbb B(G) $ and
$H=\mathcal E\odot G$ as above, then $\rho(a) :=\sigma ( a) \otimes
id_G,$ is a representation of $\mathcal A$ on $H$, mapping $x\odot
g$ to $\sigma(a)x\odot g.$
 Observe  that $L_{ax} = \rho(a)L_x.$ Also $\mathbb
B(G,H)$ forms a Hilbert $\mathcal A-\mathbb B(G)$ bi-module with left
action $ax := \rho(a)x.$ If $\mathcal E_1$ and $\mathcal E_2$ are
two Hilbert $\mathcal A-\mathcal B$ bi-modules, then a linear map $\Phi
: \mathcal E_1 \rightarrow \mathcal E_2$ is said to be $\mathcal
A-\mathcal B$-linear (or bilinear) if $\Phi(axb) = a\Phi(x)b$ for
all $a \in \mathcal A, b \in \mathcal B, x \in \mathcal E.$ The
space of all bounded, adjointable and bilinear maps from $\mathcal
E_1$ to $\mathcal E_2$ is denoted by $\textbf B^{a,
\mbox{bil}}(\mathcal E_1,\mathcal E_2).$ If $\mathcal E$ is a
Hilbert $\mathcal A-\mathcal B$ bi-module, then $\textbf B^{a,
\mbox{bil}}(\mathcal E)$ is the relative commutant of the image of
$\mathcal A$ in $\textbf B^a(\mathcal E).$

Now suppose $\mathcal B \subset \mathbb B(G)$ is a von Neumann
algebra (though out this article, by a von Neumann algebra we mean a
concrete strongly closed  $C^*$-subalgebra of $\mathbb B (G)$ for
some Hilbert space $G$ and not an abstract $W^*$-algebra). Suppose
$\mathcal E$ is a Hilbert $\mathcal B$-module. As explained before
we will identify $\mathcal E$ as a subspace of $\mathbb B(G, H)$
with $H= \mathcal E\odot G.$ Then we say $\mathcal E$ is a von
Neumann $\mathcal B$-module if $\mathcal E$ is strongly closed in
$\mathbb B(G,H) \subset \mathbb B(G \oplus H)$ (closure in the strong
operator topology (SOT)). Thus, if $x$ is an element in the strong
closure $\overline{\mathcal E}^s$ of a Hilbert $\mathcal B$-module
$\mathcal E,$ then there exists a net $(x_\alpha)$ in $ \mathcal E$
such that $L_{x_\alpha} \displaystyle{\xrightarrow{SOT}} L_x.$ All
von Neumann $\mathcal B$-modules are self-dual (in the sense that
all $\mathcal B$-valued functionals are given by a $\mathcal
B$-valued inner product with a fixed element of the module), and
hence they are complemented in all Hilbert $\mathcal B$-modules
which contain it as a $\mathcal B$-submodule. In particular,
strongly closed $\mathcal B$-submodules are complemented in a von
Neumann $\mathcal B$-module. If we use the identification $\textbf
B^a(\mathcal E) \subset \mathbb B(H),$ then $\textbf B^a(\mathcal
E)$ is a von Neumann subalgebra  of $\mathbb B(H).$

If $\mathcal A$ is a unital $C^*$-algebra, then by a von Neumann
$\mathcal A-\mathcal B$ bi-module we mean a von Neumann $\mathcal
B$-module $\mathcal E$ which is also a Hilbert $\mathcal A- \mathcal
B$ module. In particular, there is no normality assumption on the
left action $x\mapsto ax$ of $\mathcal A$ on $\mathcal E.$

Consider a von Neumann $\mathcal A-\mathcal B$ bi-module $\mathcal
E.$ If in addition, $\mathcal A$ is a von Neumann algebra
 and the left action by $\mathcal A$ is normal, then
 we call $\mathcal E$ a two-sided von Neumann $\mathcal A-\mathcal B$ bi-module. In the
 present article, we have no occasion to consider two-sided von
 Neumann modules as our domain algebra $\mathcal A$ is
 just a $C^*$-algebra and even when it is a von Neumann algebra we
 do not assume normality of left actions.
 For more details on these concepts  see (\cite{Pas},  \cite{Ske}, \cite{Sk2}).

It is well-known that if $\phi : \mathcal A \rightarrow \mathcal B$
is a completely positive map between unital $C^*$-algebras, then
there exists a Hilbert $\mathcal A-\mathcal B$-module $\mathcal E$
and $x \in \mathcal E$ such that $\phi(a) = \langle x, ax \rangle$
for all $a \in \mathcal A.$ The pair $(\mathcal E, x)$ is called a
GNS-construction for $\phi$ and $\mathcal E$ is called a GNS-module
for $\phi.$ If further, $\mathcal Ax\mathcal B:=
\overline{\mbox{span}} \{axb:a\in \mathcal A, b\in \mathcal B\}$
equals $ \mathcal E,$ then $(\mathcal E, x)$ is said to be a minimal
GNS-construction. The minimal GNS construction of $\phi $  is unique
up to isomorphism. One way to construct $\mathcal E$ is by starting
with $\mathcal A \otimes \mathcal B$ and defining a $\mathcal
B$-valued semi-inner product on it as $\langle a_1 \otimes b_1, a_2
\otimes b_2 \rangle := b^*_1\phi(a^*_1a_2)b_2,$ and carrying out
usual  quotienting and completion procedure (see \cite{Kas},
\cite{Mur}, \cite{Lan}, \cite{Pau}, \cite{Sti}).


Note that if $\mathcal B = \mathbb B(G),$ then $L^*_x\rho(a)L_x = \langle x, ax\rangle = \phi(a)$ for all
 $a\in \mathcal A.$ Thus
$(H, \rho, L_x)$ is a Stinespring representation for the CP-map $\phi : \mathcal A \rightarrow \mathbb B(G).$ Note that we may take $\mathbb B(G,H)$ as a $\mathcal A$-$\mathbb B(G)$ bi-module with right action as operator multiplication and the left action is via $\rho.$

Here we recall the following simple but important observation on Hilbert $C^*$-modules with unit vectors. This result is a particular case of Dupre-Fillmore Theorem (\cite{HM}) and also can be found in Theorem 1.4.5, \cite{HCM}.

\begin{proposition} \label{complemented}
Let $\mathcal E$ be a Hilbert $C^*$-module on a unital $C^*$-algebra
$\mathcal B.$ Suppose $\mathcal E$ has a unit element $x$ (that
is, $\langle x, x\rangle =1$), then the module $xB$ is complemented
in $\mathcal E.$
\begin{proof} Every element $y\in \mathcal{E}$ decomposes as
$$y= x.\langle x, y\rangle + [y-x.\langle x, y\rangle ],$$
and it is easily seen that this is an orthogonal decomposition of
$\mathcal E.$


\end{proof}

\end{proposition}

This result is readily applicable to minimal GNS construction of
unital completely positive maps, as the cyclic element readily
provides a unit vector.

In the complex plane and more generally in any metric space $X$, if $x\in X,$ and $B\subset X$, we take
distance between $x$ and $B$ as $d(x,B) = \inf \{ d(x,b): b\in B\}.$ The metric under
consideration should be clear from the context.

\section{The Representation metric}

Suppose $\mathcal A, \mathcal B$ are unital $C^*$-algebras. Denote
by $\mbox{UCP}(\mathcal A, \mathcal B),$  the set of all unital
completely positive maps from $\mathcal A$ to $\mathcal B.$ Let
$\phi_1,\phi_2 \in \mbox{UCP}(\mathcal A, \mathcal B).$ The Bures
distance between $\phi _1$ and $\phi _2$ is well known in the
literature. See \cite{Bur}, \cite{Ara}, \cite{Con}, \cite{BhS} for
more details. We wish to modify the setup slightly to get a new
metric. To begin with let us recall the definition of Bures
distance, using the Hilbert $C^*$-module set up \cite{BhS}.

\begin{definition}
A Hilbert $\mathcal A-\mathcal B$-module $\mathcal E$ is said to be
a {\em common representation module\/} for $\phi_1,\phi_2 \in
\mbox{UCP}(\mathcal A, \mathcal B)$ if both of them can be
represented in $\mathcal E,$ that is, there exist $x_i\in \mathcal
E$ such that $ \phi_i(a)=\langle x_i,ax_i\rangle ,  i = 1, 2.$ Then
the triple $(\mathcal E, x_1, x_2)$ is called a common
representation tuple of $(\phi_1,\phi_2 ).$
\end{definition}

Note that we are demanding no minimality for the common
representation module. So we can always have such a module. For, if
$(\mathcal E^i, x^i)$ is the minimal GNS-construction for $\phi_i$,
then take $\mathcal E = \mathcal E^1 \oplus \mathcal E^2,$ $x_1 =
x^1 \oplus 0$ and $x_2 = 0 \oplus x^2.$ For a common representation
module $\mathcal E,$ define $S(\mathcal E, \phi_i)$ to be the set of
all $x_i \in \mathcal E$ such that $\phi_i(a) = \langle x_i, ax_i\rangle$
for all $a\in \mathcal A.$ It is to be remembered that if
$\phi_1,\phi_2$ are states, then $\mathcal E$ is a Hilbert space and
$x_1, x_2$ are unit vectors in it.

\begin{definition}
Let $\mathcal E$ be a common representation module for $
\phi_1,\phi_2 \in \mbox{UCP}(\mathcal A,
 \mathcal B).$  Define $$ \beta_\mathcal E(\phi_1,\phi_2) = \mbox{inf}\{ \|x_1-x_2\|: x_i\in
 S(\mathcal E,\phi_i), i=1,2\} $$  and  $$ \beta(\phi_1,\phi_2) =\inf_\mathcal
  E~ \beta_\mathcal E(\phi_1,\phi_2) $$
 where the infimum is over all the common representation module
$\mathcal E.$
 \end{definition}

 It is shown in \cite{Con}, \cite{BhS} that $\beta $ is a metric
 when the range algebra $\mathcal B$ is a von Neumann algebra or is
 an injective $C^*$-algebra. The original definition was for states and was given by Bures in \cite{Bur}. In view of
 this we
  will call $\beta $ as the {\em Bures metric\/}. But it is important to keep in mind that the triangle inequality may fail for
 some  $\mathcal B.$ (See \cite{BhS}). Now we introduce a new
 notion.

\begin{definition}
 A Hilbert $C^*$ right $\mathcal B$ module $\mathcal E$  is said to be a {\em joint representation
 module\/}
 for $\phi_1,\phi_2 \in \mbox{UCP}(\mathcal A,\mathcal B)$ if there are two unital left $\mathcal A$ actions $\sigma_1$
 and $\sigma_2$ and a unital vector $x \in \mathcal E$ such that $$ \phi_1(a) = \langle x,
  \sigma_1(a) x \rangle, ~~ \phi_2(a) = \langle x, \sigma_2(a) x \rangle. $$ Here $
   (\mathcal {E},   \sigma _1, \sigma _2, x)$ is called a joint representation tuple
   of $(\phi _1, \phi _2).$
\end{definition}

In the following by `$\| \cdot \|_{cb}$' we mean the completely
bounded (CB) norm. We are taking the CB norm of differences of two
$*$-homomorphisms. Subsequent computations in the article show that
CB norm (and not the usual norm) is the right choice in the present
context.

\begin{definition}\label{Stinespring}
Let $\phi_1,\phi_2 \in \mbox{UCP}(\mathcal A, \mathcal B).$ Let
$\tilde {\mathcal E}= (\mathcal {E},   \sigma _1, \sigma _2, x)$ be
a joint representation tuple. Define
 $$\gamma ^{\tilde {\mathcal E}} (\phi_1,\phi_2)=\|\sigma_1-\sigma_2\|_{cb}.$$
  Define   $$ \gamma (\phi_1,\phi_2) := \inf_{(\mathcal E , \sigma _1, \sigma _2, x)}\| \sigma _1-\sigma _2\|_{cb} ,$$
  where the infimum is taken over all joint representation tuples.
\end{definition}
We will informally call $\gamma (\phi _1, \phi _2)$ as the {\em representation metric} or {\em representation distance} between $\phi _1, \phi _2$. It will be seen that under good situations it is
indeed a metric analogues to Bures metric. This notion can also be compared with Wasserstein metric.

For a Hilbert $\mathcal A-\mathcal B$ bi-module $\mathcal E$ with
left action $\sigma $ and $x\in \mathcal E,$ define $$\mathcal
A_\sigma x\mathcal B=\overline{\mbox{span}}~\{\sigma(a)xb: a\in
\mathcal A, b\in \mathcal B\}.$$ If the left action is clear from
the context, we denote it by $\mathcal A x\mathcal B.$

For a joint representation tuple $(\mathcal E,\sigma_1,\sigma_2,x),$ we define $$\mathcal A_{\sigma_1,\sigma_2}x\mathcal B=\overline{\mbox{span}}~\{\sigma_{\epsilon_1}(\mathcal A)\sigma_{\epsilon_2}(\mathcal A)\cdots\sigma_{\epsilon_k}(\mathcal A)x\mathcal B: \epsilon_i=1 ~\mbox{or}~2,k\geq 1\}.$$

\begin{definition}
A joint representation tuple $(\mathcal E,  \sigma _1, \sigma _2, x)$ is said to be {\em minimal}  if $\mathcal A_{\sigma_1,\sigma_2}x\mathcal B=\mathcal E.$
\end{definition}

\begin{remark}\label{remark Stinespring}
It  suffices to consider  minimal joint representation tuples while
taking the infimum  in Definition \ref{Stinespring}.
\end{remark}

Denote by $J(\phi_1,\phi_2)$,  the set of all joint representation
tuples for $\phi_1$ and $\phi_2.$

\begin{proposition}\label{non-empty} Let $\phi_1,\phi_2 \in \mbox{UCP}(\mathcal A, \mathcal B).$ Then
$J(\phi_1,\phi_2)$ is non-empty.
\end{proposition}

\begin{proof}
Let $(\mathcal E_i, x_i)$ be the minimal GNS construction for $\phi_i,$ with $\tau_i$ be the left action, $i=1,2.$  Take
$\mathcal E=\mathcal E_1\oplus\mathcal E_2.$ Then $\mathcal E$ is an
$\mathcal A-\mathcal B$ bi-module with left action $\sigma _1= \tau
_1\oplus \tau _2.$ By Proposition \ref{complemented}, we have
orthogonal decompositions, $\mathcal E _i= x_iB\oplus \mathcal E ^0
_i$, for $i=1,2.$ Hence $\mathcal E= x_1B\oplus \mathcal E ^0
_1\oplus x_2B\oplus \mathcal E ^0 _2.$ Define $U:\mathcal E\to
\mathcal E$, by
$$U(x_1b_1\oplus y_1\oplus x_2b_2\oplus y_2)= x_1b_2\oplus y_1\oplus x_2b_1\oplus y_2 ~~~\forall b_1,b_2\in \mathcal B, y_i\in \mathcal E ^0 _i, i=1,2.$$
It is easily seen that $U$ is a right $\mathcal B$-linear unitary and $U(x_1\oplus 0 )=  0\oplus x_2$ in $\mathcal E =\mathcal E _1\oplus \mathcal E _2.$
 Define $\sigma_2: \mathcal A\rightarrow \mathcal B^a(\mathcal E) $ by $\sigma_2(a)=U^*\sigma_1(a)U.$  Note that $\phi_1(a)=\langle x_1, \sigma_1(a)x_1\rangle $ and $\phi_2(a)=\langle x_1,\sigma_2(a)x_1\rangle.$ It follows that $(\mathcal E, x_1\oplus 0 , \sigma_1,\sigma_2)$ is a joint representation tuple for $\phi_1,\phi_2.$
\end{proof}

The following remark is important for the proof of next theorem.

\begin{remark}
 Suppose $\mathcal B\subseteq B(G)$ is a von Neumann algebra
and suppose $\mathcal E$ is a Hilbert $C^*$ right $\mathcal
B$-module. Now considering the SOT closure  $\overline{\mathcal
E}^s$ of $\mathcal E\subseteq \mathbb B (G,H)$ we get a von Neumann
$ \mathcal B$-module. This won't effect the left actions of
$\mathcal A$, as they are unital $*$-representations taking values
in $\mathbb B(H).$ So if $\mathcal E$ is a Hilbert $\mathcal
A-\mathcal B$ bi-module then $\overline {\mathcal E}^s$ becomes a
von Neumann $\mathcal A-\mathcal B$ bi-module.  Then it follows
easily that
$$ \gamma(\phi_1,\phi_2)= \inf_\mathcal E \gamma_\mathcal
E(\phi_1,\phi_2),$$ where the infimum is over all joint
representation modules $\mathcal E$ which are von Neumann $\mathcal
A-\mathcal B$ modules. It is to be noted that $\mathcal A$ can be a
general unital $C^*$-algebra and left actions are not assumed to be
normal. So we do not need $\phi_1,\phi_2$ to be normal maps.
\end{remark}

\begin{theorem}\label{metric}
Suppose $\mathcal B\subset \mathbb B(G)$ is a von Neumann algebra.
Then $\gamma$ is a metric on $\mbox{UCP}(\mathcal A, \mathcal B).$
\end{theorem}

\begin{proof}

It is evident that $\gamma(\phi,\phi)=0$ and
$\gamma(\phi_1,\phi_2)=\gamma(\phi_2,\phi_1).$ Note that given any
joint representing module $\mathcal E,$ $\|
\sigma_1-\sigma_2\|^{\mathcal A \rightarrow \textbf B^a(\mathcal
E)}_{cb} \geq \|\phi_1-\phi_2\|_{cb},$ which shows that
$\gamma(\phi_1,\phi_2)\geq \|\phi_1-\phi_2\|_{cb}.$ So
$\gamma(\phi_1,\phi_2)=0$ implies $\phi_1=\phi_2.$ It remains to
show the triangle inequality. Suppose $\phi_1,\phi_2,\phi_3 \in
\mbox{UCP}(\mathcal A, \mathcal B).$ Let $\epsilon > 0$ be given, find von
 Neumann modules $(\mathcal E_1,\sigma_1,\sigma_2, x_1) \in J(\phi_1,\phi_2)$ and
$(\mathcal E_2, \sigma^\prime_2,\sigma^\prime_3, x_2)\in
J(\phi_2,\phi_3)$ such that $\|\sigma_1-\sigma_2\|_{cb} <
\gamma(\phi_1,\phi_2)+\frac{\epsilon}{2}$ and
$\|\sigma^\prime_2-\sigma^\prime_3\|_{cb} < \gamma(\phi_2,\phi_3)+
\frac{\epsilon}{2}.$ There exists bilinear unitary $W$
from the Hilbert $C^*$ sub-module $\mathcal A_{\sigma_2}x_1\mathcal B$
to $\mathcal A_{\sigma^\prime_2}x_2\mathcal B$ given by
$W(\sigma_2(a)x_1b)=\sigma^\prime_2(a)x_2b.$ Note that $\langle Wx,Wy\rangle=\langle x,y\rangle$ for every
 $x,y\in \mathcal A_{\sigma_2}x_1\mathcal B.$ Therefore $ \langle W(x)\odot g, W(y)\odot g^\prime\rangle =
 \langle x\odot g,y\odot g^\prime\rangle $ for every $x,y\in \mathcal A_{\sigma_2}x_1\mathcal B$ and $g,g^\prime\in G.$
  This implies that $W$ extends to a map from the von Neumann module $\overline{\mathcal A_{\sigma_2}x_1\mathcal B}^s$
to $\overline{\mathcal A_{\sigma^\prime_2}x_2\mathcal B}^s.$ Let us
call this extension again by $W.$ Note that $\overline{\mathcal
A_{\sigma_2}x_1\mathcal B}^s$ and $\overline{\mathcal
A_{\sigma^\prime_2}x_2\mathcal B}^s$ are $\mathcal B$ modules as
$\mathcal B$ is
 strong operator closed in $\mathbb B(G).$ Define right $\mathcal
B$ modules: $$\mathcal E=\mathcal E_1 \oplus (\overline{\mathcal
A_{\sigma^\prime_2} x_2\mathcal B}^s)^{\bot} = (\overline{\mathcal
A_{\sigma_2}x_1\mathcal B}^s)^{\bot} \oplus \overline{\mathcal
A_{\sigma_2}x_1\mathcal B}^s \oplus (\overline{\mathcal A_{\sigma^\prime_2}
x_2\mathcal B}^s)^{\bot},$$ $$ \mathcal E^\prime = \overline{(\mathcal
A_{\sigma_2}x_1\mathcal B}^s)^{\bot} \oplus \mathcal E_2 = (\overline{\mathcal
A_{\sigma_2}x_1\mathcal B}^s)^{\bot} \oplus \overline{\mathcal
A_{\sigma^\prime_2} x_2\mathcal B}^s \oplus (\overline{\mathcal
A_{\sigma^\prime_2} x_2\mathcal B}^s)^{\bot}.$$ In these modules, with
natural identifications, we have $\tilde {x}_1:=x_1\oplus 0=0\oplus
x_1\oplus 0$ and $\tilde {x}_2=0\oplus x_2=0\oplus x_2\oplus 0.$
Consider left actions defined as follows :  $$ \tilde{\sigma}_1 :=
\sigma_1\oplus \sigma^\prime_2,~ \tilde{\sigma}_2:=
\sigma_2\oplus\sigma^\prime_2 ~\mbox{acting ~on} ~ \mathcal E,$$
$$ \tilde{\tilde{\sigma}}_3 := \sigma_2 \oplus \sigma^\prime_3, ~ \tilde{\tilde{\sigma}}_2 := \sigma_2 \oplus \sigma^\prime_2 ~ \mbox{acting~on}~\mathcal E^\prime. $$ The unitary $W$ extends to an adjointable (right $B$ linear) unitary map $W^\prime: \mathcal E\rightarrow \mathcal E^\prime $ by defining
$W'=I\oplus W\oplus I.$ Observe that $\tilde{\sigma}_2(\cdot)=
W^{\prime *}\tilde{\tilde{\sigma}}_2(\cdot)W^\prime.$ Consider  left
actions  $\tilde{\sigma}_1(\cdot)$ and $\hat {\sigma }_3(\cdot
):=W^{\prime *}\tilde{\tilde{\sigma_3}}(\cdot)W^\prime$ on $\mathcal
E$ together with $x_1\in \mathcal E.$ Note that $$\langle \tilde {x}
_1, \tilde{\sigma_1}(a)\tilde {x} _1 \rangle = \langle x_1,
\sigma_1(a)x_1 \rangle = \phi_1(a) ,$$ and
\begin{eqnarray*} \langle \tilde {x} _1, \hat {\sigma }_3(a)\tilde {x} _1 \rangle &=& \langle \tilde {x}_1, W^{\prime *}\tilde{\tilde{\sigma}}_3(a)W^\prime \tilde {x} _1 \rangle \\
&=& \langle \tilde {x}_2, \tilde{\tilde{\sigma}}_3(a)\tilde {x} _2 \rangle \\ &=& \langle x_2, \sigma^\prime_3(a)x_2 \rangle \\ &=& \phi_3(a).  \end{eqnarray*} This shows that $(\mathcal E,  \tilde{\sigma}_1, \hat {\sigma } _3, \tilde {x}_1)$ is a joint representation tuple for $\phi_1,\phi_3.$ Note also $\|\sigma_1-\sigma_2\|_{cb}=\|\tilde{\sigma}_1-\tilde{\sigma}_2\|_{cb}$ and $\| \sigma^\prime_2-\sigma^\prime_3\|_{cb}=\|\tilde{\tilde{\sigma
}}_2-\tilde{\tilde{\sigma}}_3\|_{cb}.$
Now \begin{eqnarray*} \| \tilde {\sigma }_1- \hat {\sigma }_3\| _{cb} &= & \| \tilde{\sigma}_1 - W^{\prime *}\tilde{\tilde{\sigma}}_3W^\prime \|_{cb}\\
 &\leq & \| \tilde{\sigma}_1 - \tilde{\sigma}_2 \|_{cb} + \|\tilde{\sigma}_2-  W^{\prime *}\tilde{\tilde{\sigma}}_3W^\prime  \|_{cb} \\ &=& \|\sigma_1-\sigma_2 \|_{cb}
 + \|  W^{\prime }{\tilde{\sigma}}_2W^{\prime *}-\tilde{\tilde{\sigma}}_3\|_{cb} \\
  &=& \|\sigma_1-\sigma_2 \|_{cb} + \| \tilde{\tilde{\sigma}}_2-\tilde{\tilde{\sigma}}_3\|_{cb} \\ &=& \|\sigma_1-\sigma_2 \|_{cb} + \|\sigma^\prime_2-\sigma^\prime_3\|_{cb} \\ & < & \gamma(\phi_1,\phi_2) + \gamma(\phi_2,\phi_3) + \epsilon. \end{eqnarray*} As $\epsilon >0$ is arbitrary, we get  $\gamma(\phi_1,\phi_3)\leq \gamma(\phi_1,\phi_2)+\gamma(\phi_2,\phi_3).$

\end{proof}
In the next section we will show that $\gamma $ is a metric whenever
the range $C^*$-algebra is injective. The following proposition says
that representation metric is stable under taking ampliations.

\begin{proposition}\label{amplifiation}
Let $\mathcal A$ be a unital $C^*$-algebra and $\mathcal B$ be a von
Neumann algebra. Let $\phi,\psi \in \mbox{UCP}(\mathcal A,\mathcal
B).$ Then $$\gamma(\phi,\psi)=\gamma(\phi^{(n)},\psi^{(n)})$$ where
$\phi^{(n)},\psi^{(n)}:M_n(\mathcal A)\rightarrow M_n(\mathcal B)$ are ampliations of $\phi$ and $\psi$ respectively for
$n\geq 1.$
\end{proposition}

\begin{proof}

Fix $n\geq 1.$ Suppose $(\mathcal E,\sigma_\phi,\sigma_\psi,x)$ is a joint representation tuple for $\phi$ and $\psi.$ Then $M_n(\mathcal E)$ is an $M_n(\mathcal A)-M_n(\mathcal B)$ bi-module. Denote $\textbf{x}=\mbox{diag}~(x,x,\cdots,x)\in M_n(\mathcal E).$ Let $(\sigma_\phi)^{(n)},(\sigma_\psi)^{(n)}$ be ampliations of $\sigma_\phi$ and $\sigma_\psi$ respectively. Then $$(M_n(\mathcal E), (\sigma_\phi)^{(n)},(\sigma_\psi)^{(n)},\textbf{x})$$ is a joint representation tuple for $\phi^{(n)}$ and $\psi^{(n)}.$ Note that $$\|(\sigma_\phi)^{(n)}-(\sigma_\psi)^{(n)}\|_{cb}=\|\sigma_\phi-\sigma_\psi\|_{cb}.$$ As $(\mathcal E,\sigma_\phi,\sigma_\psi,x)$ is an arbitrary joint representation tuple for $\phi$ and $\psi,$ we get $\gamma(\phi^{(n)},\psi^{(n)})\leq \gamma(\phi,\psi).$ Conversely suppose $(\mathcal F,\tau_1,\tau_2,y)$ is a joint representation tuple for $\phi^{(n)}$ and $\psi^{(n)}.$ Let $(e_{ij})$ and $(f_{ij})$ $1\leq i,j\leq n$ be the matrix units of $M_n(\mathcal A)$ and $M_n(\mathcal B)$ respectively. Then $\mathcal Ff_{11}$ is an $\mathcal A-\mathcal B$ bi-module with actions $\sigma_i(a)=\tau_i(ae_{11}),$ $i=1,2.$ Let $x=yf_{11}.$ Then $(\mathcal Ff_{11},\sigma_1,\sigma_2,x)$ is a joint representation tuple for $\phi$ and $\psi.$ Note that $$\|\sigma_1-\sigma_2\|_{cb} \leq \|\tau_1-\tau_2\|_{cb}.$$ As  $( \mathcal F,\tau_1,\tau_2,y )$ is an arbitrary joint representation tuple for $\phi^{(n)}$ and $\psi^{(n)},$ we get $\gamma(\phi,\psi)\leq \gamma(\phi^{(n)},\psi^{(n)}).$

\end{proof}

\section{Relation to free products}

Suppose $\mathcal C,\mathcal D$ are two unital $C^*$-algebras.
Denote by $\mathcal C\circ \mathcal D$ the unital $*$-algebra of all
finite linear combinations of all possible finite words consisting of
elements of $\mathcal C$ and $\mathcal D.$ Define a norm on this
 algebra by $$ \|c\| = \mbox{sup}~\{\|\pi(c)\|: \pi~\mbox{is~a} ~*\mbox{-representation~of}~\mathcal C\circ
\mathcal D~\mbox{on~some~Hilbert space}~H\}. $$ This is a $C^*$
norm. Completion of $\mathcal C\circ \mathcal D$ under this norm is
called the full free product of $\mathcal C$ and $\mathcal D$ and is
denoted by $\mathcal C*\mathcal D.$

 We have canonical injections $\rho_{\mathcal C}:\mathcal C\rightarrow \mathcal C*\mathcal D,$  $
\rho_{\mathcal D}:\mathcal D\rightarrow \mathcal C*\mathcal D$. This
way, $\mathcal C, \mathcal D$ are considered as sub-algebras of
$\mathcal C *\mathcal D.$ Any $*$-representation of $\mathcal C *
\mathcal D$ on a Hilbert space $H$  restricts to a  pair of
$*$-representations of $\mathcal C, \mathcal D$. Conversely any
pairs of $*$-representations of $\mathcal C$ and $\mathcal D$ on a
common Hilbert space $H$ can be extended to a representation of
$\mathcal C*\mathcal D.$ This follows from the universal property of
the full free product. Thus  there is a 1-1 correspondence between
the $*$-representations of $\mathcal C * \mathcal D$ and  pairs of
$*$-representations of $\mathcal C$ and $\mathcal D$ on a common
Hilbert space $H.$

Let $\mathcal A* \mathcal A$ be the free product of $\mathcal A$
with itself. Let $\rho_1, \rho_2$ be the canonical injections. Let
$\phi_1,\phi_2\in \mbox{UCP}(\mathcal A,\mathcal B).$ Denote by
$$K(\phi_1,\phi_2)=\{\phi\in \mbox{UCP}(\mathcal A*\mathcal A,
\mathcal B): \phi\circ\rho_1=\phi_1,\phi\circ\rho_2=\phi_2 \}.$$ A
map in $K(\phi _1, \phi _2)$ is like  a bivariate  distribution with
given marginals. This helps us to show that the metric $\gamma $ is
somewhat like the Wasserstein metric for probability measures.

\begin{remark}\label{1-1 free}
There is a 1-1 correspondence between the set of all Hilbert $C^*$
right $\mathcal B$  modules $\mathcal E$ with left actions
$\sigma_1,\sigma_2$ and the set of all $\mathcal A*\mathcal
A-\mathcal B$ bi-modules $(\mathcal E,\sigma).$ Indeed, for an
$\mathcal A*\mathcal A-\mathcal B$ bi-module $\mathcal E$ letting
$\sigma_i=\sigma\circ\rho_i,$ $i=1,2,$ we may endow $\mathcal E$
with two left actions $\sigma_1,\sigma_2.$ Conversely given a module
$(\mathcal E, \sigma_1,\sigma_2),$ the universal property of
$\mathcal A*\mathcal A$ provides $\sigma:\mathcal A*\mathcal
A\rightarrow \textbf B^a(\mathcal E)$ satisfying
$\sigma\circ\rho_i=\sigma_i,i=1,2.$ By virtue of the fact above,
every joint representation module $(\mathcal E,\sigma_1,\sigma_2,x)$
corresponds uniquely to an $\mathcal A*\mathcal A-\mathcal B$
bi-module $(\mathcal E,x).$ Also the joint representation module is
minimal if and only if $(\mathcal A*\mathcal A) x \mathcal
B=\mathcal E.$
\end{remark}

\begin{theorem}\label{1-1}
Let $\mathcal A, \mathcal B$ be unital $C^*$-algebras and let $\phi
_i\in \mbox{UCP}(\mathcal A, \mathcal B)$ for $i=1,2$. Then there is
a 1-1 correspondence  between the set of
 minimal joint representation modules of $(\phi_1,\phi_2)$ (modulo isomorphism)  and the set
of maps in  $K(\phi_1,\phi_2).$
\end{theorem}

\begin{proof}  Suppose $(\mathcal E,\sigma_1,\sigma_2,x)$ is a minimal joint representation module
for $\phi_1$ and $\phi_2.$ By Remark \ref{1-1 free}, we may consider
$(\mathcal E,x)$ an $(\mathcal A*\mathcal A)-\mathcal B$ bi-module
with left action $\sigma$ (say). We associate a completely positive
map $\Phi((\mathcal E,\sigma_1,\sigma_2,x)):=\phi\in
K(\phi_1,\phi_2)$ by $\phi(c)=\langle x,\sigma(c)x\rangle.$
Conversely every element in $K(\phi_1,\phi_2)$ under minimal GNS
construction gives rise to the minimal $\mathcal A*\mathcal
A-\mathcal B$ bi-module $\mathcal E$ and a unital vector $x\in
\mathcal E.$ So by Remark \ref{1-1 free}, we get a minimal joint
representation module of $\phi_1$ and $\phi_2.$ From the uniqueness
of minimal dilation, it follows that these two operations are
inverses of  each other. Indeed, given a representation module
$(\mathcal E,\sigma_1,\sigma_2,x),$ consider $\phi=\Phi((\mathcal
E,\sigma_1,\sigma_2,x)).$ Let $\mathcal E^\prime$ be its minimal
Stinespring bi-module. Define $\Psi:\mathcal E^\prime \rightarrow
\mathcal E$ by $\Psi(c\overline{\otimes} b)=cxb,$ $c\in \mathcal
A*\mathcal A,$ $b\in \mathcal B.$ By definition $\Psi$ is an
bi-module isometry and from the minimality, it follows that $\Psi$
is onto. Therefore $\Psi$ is an bi-linear unitary. Hence
$\Phi^{-1}(\Phi((\mathcal E,\sigma_1,\sigma_2))\simeq \mathcal E.$
The other part is trivial.

\end{proof}

We observe that Theorem \ref{1-1} gives us a different method of
computing $\gamma (\phi _1, \phi _2).$  In view of  Remark
\ref{remark Stinespring}, in the definition of $\gamma $ it  is
enough to consider  minimal joint representation modules.  Now by
the previous theorem it suffices to consider GNS modules of $\phi\in
K(\phi_1,\phi_2).$ Let $(\mathcal E,x)$ be the minimal GNS
construction of $\phi.$ i.e. $\phi(c)=\langle x,cx \rangle.$ Note
that $\textbf B^a(\mathcal E)$ is a $C^*$-algebra. Let
$\rho_1,\rho_2$ be canonical injections from $\mathcal A$ to
$\mathcal A*\mathcal A.$ Then the left action $\sigma$ of $\mathcal
E$ induces homomorphisms $\sigma_i=\sigma\circ\rho_i:\mathcal
A\rightarrow \textbf B^a(\mathcal E),i=1,2.$ Then the computation of
representation metric can be done as $$
\gamma(\phi_1,\phi_2)=\inf_{\phi\in
K(\phi_1,\phi_2)}\{\|\sigma_1-\sigma_2\|^\mathcal E_{cb}: (\mathcal
E,x)~\mbox{is~the~minimal~GNS~module~of}~\phi\}. $$


As an application of these ideas we get the following result.

\begin{proposition}\label{composition}
Let $\mathcal A,\mathcal B$ and $\mathcal C$ be unital $C^*$-algebras. Let
$\phi_1,\phi_2\in \mbox{UCP}(\mathcal A, \mathcal B)$ and
$\psi\in \mbox{UCP}(\mathcal B,\mathcal C).$ Then
$\gamma(\psi\circ\phi_1,\psi\circ \phi_2) \leq
\gamma(\phi_1,\phi_2).$
\end{proposition}

\begin{proof}

Given $\phi\in K(\phi_1,\phi_2),$ observe  that $\psi\circ\phi\in
K(\psi\circ\phi_1,\psi\circ\phi_2).$ Let ($\mathcal E_\phi,x)$ be a
$\mathcal A*\mathcal A-\mathcal B$ bi-module. Let
$\rho_1,\rho_2:\mathcal A \rightarrow \textbf B^a(\mathcal E_\phi)$
be the canonical maps so that $\langle x,\rho_i(\cdot )x\rangle =
\phi_i (\cdot ),$ $i=1,2.$ Consider the completely positive map
$\tilde{\psi}:\textbf B^a(\mathcal E)\rightarrow \mathcal C$ defined
by $\tilde{\psi}(A)=\psi(\langle x, Ax\rangle),$ $A\in \textbf
B^a(\mathcal E ).$ Let $(\mathcal E_{\tilde{\psi}},y)$ be a $\textbf
B^a(\mathcal E)-\mathcal C$ GNS bi-module for $\tilde{\psi}.$ Denote
by $\tilde{\pi}$ be its corresponding left action. We have $ \langle
y, \tilde{\pi}(\rho_i(a))y\rangle = \psi\circ \phi_i,$ $i=1,2.$ We
get $(\mathcal E_{\tilde{\psi}},y, \tilde{\pi}\circ
\rho_1,\tilde{\pi}\circ \rho_2)$ is a joint representation tuple for
$(\psi\circ \phi_1,\psi\circ \phi_2).$ Therefore
$\gamma(\psi\circ\phi_1,\psi\circ\phi_2)\leq \|\tilde{\pi}\circ
\rho_1-\tilde{\pi}\circ \rho_2\|_{cb} \leq
\|\rho_1-\rho_2\|^{\mathcal E_\phi}_{cb}.$ Taking infimum over
$\phi\in K(\phi_1,\phi_2),$ the result follows.

\end{proof}

\begin{theorem}\label{injective}
Let $\mathcal A$ be a $C^*$-algebra and let $\mathcal B\subseteq
\mathbb B(G)$ be an injective $C^*$-algebra. Let $\iota: \mathcal
B\rightarrow \mathbb B(G)$ be the inclusion map. Then for any two
UCP maps $\phi _i: \mathcal A\to \mathcal B$, taking
$\tilde{\phi}_i:=\iota\circ\phi_i,$ $i=1,2,$:
$$ \gamma(\tilde{\phi}_1,\tilde{\phi}_2) = \gamma(\phi_1,\phi_2). $$
Consequently, $\gamma $ is a metric on UCP($\mathcal A, \mathcal
B).$
\end{theorem}

\begin{proof}
As $\mathcal B$ is injective, there exists a conditional expectation
map  $\Phi:\mathbb B(G)\rightarrow \mathcal B$ (this means that,
$\Phi $ is UCP and $\Phi (b)=b$ for $b\in \mathbb B$). Note that in
particular, $\phi_i=\Phi\circ\tilde{\phi_i},$ $i=1,2.$ Now from
Proposition \ref{composition}, we get
\begin{eqnarray*}
 \gamma(\phi_1,\phi_2) &=& \gamma(\Phi\circ\tilde{\phi}_1,\Phi\circ\tilde{\phi}_2) \\
 &\leq& \gamma(\tilde{\phi}_1,\tilde{\phi}_2)\\ &=& \gamma(\iota\circ\phi_1,\iota\circ\phi_2)\\ &\leq& \gamma(\phi_1,\phi_2).
  \end{eqnarray*}
  Now the second part follows from Theorem \ref{metric}.
\end{proof}

\section{Attainability of the metric}

The representation metric is defined as an infimum of completely
bounded norm of differences of a class of $*$-homomorphisms. It is a
natural question as to whether the infimum is actually attained at
some pair. In this section, we will address this issue. Suppose
$\phi_1,\phi_2\in \mbox{UCP}(\mathcal A,\mathcal B)$ where $\mathcal
A$ is a unital $C^*$-algebra and $\mathcal B$ is a von Neumann
algebra. Suppose $\mathcal B\subseteq \mathbb B(G)$ for some Hilbert
space $G.$ Recall the definition of $K(\phi _1, \phi _2)$ from the
previous Section. From Proposition \ref{non-empty}, Theorem
\ref{1-1}, we see that
$K(\phi_1,\phi_2)$ is non-empty. Our first observation is that the space $K(\phi _1, \phi _2)$ is compact under suitable topology.

Let $\mathcal C$ be a $C^*$-algebra. Fix
$r> 0.$ Let  us recall BW (bounded weak) topology on
$\mbox{CP}_r(\mathcal{C}, \mathbb{B}(G))=\{ \phi \in \mbox{CP}(\mathcal C, \mathbb B(G)) : ~~\|\phi \| \leq r\}.$  A net $\phi
_\alpha \to \phi $ in BW topology if  for every $c\in \mathcal{C},
\xi, \mu \in G$ $\langle \xi , (\phi _{\alpha }(c)-\phi
(c))\mu \rangle \to 0.$ It is to be noted that $\mbox{CP}_r(C,
\mathbb{B}(G))$ is compact with respect to BW topology. As
$\mathcal{B}$ is a von Neumann algebra, it follows that $\mbox{CP}_r(C,
\mathcal{B})=\{ \phi\in \mbox{CP}(\mathcal C,\mathcal B) : \|\phi \| \leq r\}$ is a closed subset of $\mbox{CP}_r(C, \mathbb{B}(G))$
in BW topology and hence compact. Consequently $K(\phi _1, \phi _2)$ being closed subset of $\mbox{CP}_r(\mathcal A*\mathcal A,
\mathcal{B})$ is BW-compact.

 Consider $\phi \in K(\phi _1, \phi _2).$  Let $(\mathcal E,x)$ be its minimal GNS construction. Then  $\mathcal E$ is an
$\mathcal A*\mathcal A-\mathcal B$ bi-module and  also it is a von
Neumann right $B$ module. Note that $G$ is a $\mathcal B-\mathbb C$
bi-module. Consider the internal tensor product  $H=\mathcal E\odot
G.$ Note that $H$ is a Hilbert space and $\textbf B^a(\mathcal E)$
is a von Neumann sub-algebra of $\mathbb B(H).$ Let $\sigma:\mathcal
A*\mathcal A \rightarrow \textbf B^a(\mathcal E)$ be the unital left
action and let $\rho_1,\rho_2:\mathcal A\rightarrow \mathcal
A*\mathcal A$ be the canonical injections. Suppose
$\sigma_i=\sigma\circ\rho_i,$ $i=1,2.$  For notational simplicity we
are suppressing the dependence of $\sigma _1, \sigma _2$ on $\phi $.
However, we will denote the completely bounded norm of $\sigma
_1-\sigma _2$, by
 $\| \sigma _1-\sigma _2\|^{\phi }_{cb}.$   Recall that
$$\gamma (\phi _1, \phi _2)=  \inf _{\phi\in K(\phi_1,\phi_2) } \| \sigma _1-\sigma _2\|^{\phi } _{cb}.$$ Hence we need to study the behaviour of the map
$\phi \mapsto \| \sigma _1-\sigma _2\|^{\phi } _{cb}$ under BW topology. As $\sigma _1, \sigma _2$ are $*$-homomorphisms, $\|\sigma _1-\sigma _2\| ^{\phi } _{cb}\leq 2$. From the definition of norm,
 \begin{eqnarray*} && \|\sigma_1-\sigma_2\| \\ &=& \sup_{\|a\|\leq 1,~ a\in \mathcal A}~\|(\sigma_1(a)-\sigma_2(a))^*(\sigma_1(a)-\sigma_2(a))\|^{\frac{1}{2}} \\
 &=& \sup_{a\in \mathcal A, \|a\|\leq 1  } ~~ \sup_{\eta \in \mathcal E \circ G, \|\eta\| \leq 1} [\langle \eta, [\sigma_1(a^*a)+
 \sigma _2(a^*a)-2\mbox {Re} (\sigma _1(a^*)\sigma _2(a))]\eta \rangle]^{\frac{1}{2}}.  \end{eqnarray*}

By minimality of the Stinespring dilation,  $H=\mathcal{E}\odot G =
\overline{\mbox{span}}\{ \sigma (c)xb\odot g: c\in \mathcal{C}, b\in
\mathcal{B}, g\in G\}.$ Hence the collection of vectors of the form
$\eta = \sum _{i=1}^k \sigma (c_i)xb_i\odot g_i$ is dense in $H.$
Now,

 $$ \|\eta\|^2= \sum_{i,j} \langle b_jg_j, \phi(c^*_jc_i)b_ig_i \rangle $$
 and
 \begin{eqnarray*} && \langle \eta, [\sigma_1(a^*a)+\sigma_2(a^*a)-2 \mbox{Re}(\sigma_1(a^*)\sigma_2(a))]\eta \rangle \\   &=&  \sum_{i,j} \langle (b_jg_j), \phi(c^*_j(\rho _1(a^*a)+\rho _2(a^*a)
 -2~\mbox{Re}(\rho _1(a^*)\rho _2(a))c_i)(b_ig_i) \rangle \end{eqnarray*}

  Denote by $$\tilde{c}=(c_1,c_2,\cdots, c_k), ~ \tilde{b}=(b_1,b_2,\cdots,b_k), ~ \tilde{g}=(g_1,g_2,\cdots,g_k).$$ Define \begin{eqnarray*}  && f(\phi,k,a,\tilde{c},\tilde{b},\tilde{g}) \\ &=& \frac{[\sum^k_{i,j}\langle (b_jg_j), \phi(c^*_j(\rho_1(a^*a)+\rho _2(a^*a)
  -2~\mbox {Re}(\rho _1(a^*)\rho _2(a)))c_i)(b_ig_i) \rangle]^\frac{1}{2}}{[\sum^k_{i,j} \langle b_jg_j, \phi(c^*_jc_i)b_ig_i \rangle]^\frac{1}{2}}.\end{eqnarray*} Note that numerator vanishes if denominator vanishes, and in such a case this ratio is defined to be 0. Observe that $\phi \mapsto
   f(\phi,k,a,\tilde{c},\tilde{b},\tilde{g})$ is continuous in BW topology, when other variables are kept fixed.
   Also note that $f(k,a,\tilde{c},\tilde{b},\tilde{g})$ is bounded by $2\|a\|.$ Therefore $$ \|\sigma_1-\sigma_2\|^{\phi} = \sup_{k\in {\mathbb N},\|a\| \leq 1,\tilde{c},\tilde{b},\tilde{g}} ~f(\phi,k,a,\tilde{c},\tilde{b},\tilde{g}).$$

  In order to compute the completely bounded norm of $\sigma _1-\sigma _2$,  we need to consider, $M_n(\mathcal A),$ $\hat{\eta}=(\eta_1,\eta_2,\cdots,\eta_n)\in H \oplus \cdots \oplus H$ ($n$ times) and $\phi$ to be replaced by $ \phi ^{(n)}$ (ampliation of $\phi$). It follows that $$ \|\sigma_1-\sigma_2\|^{\mathcal E_\phi}_{cb} = \sup_{\|(a_{ij})\| \leq 1, n, k, \tilde{c}_i,\tilde{b}_i,\tilde{g}_i,1\leq i\leq n}~ F(\phi,n,k, (a_{ij}),(\tilde{c}_i),(\tilde{b}_i),(\tilde{g}_i))$$ where $\tilde{c}_i=(c_{i1},c_{i2},\cdots, c_{ik}),$ $\tilde{b}_i=(b_{i1},b_{i2},\cdots,b_{ik}),$ $\tilde{g}_i=(g_{i1},g_{i2},\cdots,g_{ik}).$ Then $\eta_i=\sum^k_{j=1} c_{ij}x b_{ij}\circ g_{ij},$ $i=1,2,\cdots,n$ and \begin{eqnarray*} F(\phi,n,k,(a_{ij}),(\tilde{c}_i),(\tilde{b}_i),(\tilde{g}_i)) &=& \frac{[\sum^n_{i=1,j=1,l=1}\sum^k_{r=1,r^\prime=1} A_{ijlrr^\prime}]^\frac{1}{2}}{[\sum^n_{i=1}\sum^k_{r=1,r^\prime=1}\langle b_{ir}g_{ir}, \phi(c^*_{ir}c_{ir^\prime})b_{ir^\prime}g_{ir^\prime}\rangle]^\frac{1}{2}}, \\ \end{eqnarray*} where $$A_{ijlrr^\prime}=
  \langle b_{ir}g_{ir},\phi(c^*_{ir}(\sigma_1(a^*_{il}a_{lj})+\sigma_2(a^*_{il}a_{lj})-2~\mbox {Re}\sigma_1(a^*_{il})\sigma_2(a_{lj}))c_{jr^\prime})b_{jr^\prime}g_{jr^\prime} \rangle. $$
 Once again it is easy to see  that  $\phi \rightarrow F(\phi,n,k, (a_{ij}),(\tilde{c}_i),(\tilde{b}_i),(\tilde{g}_i))$ is BW continuous when other quantities are kept fixed. Now we are ready to prove the following lemma.


\begin{lemma}

Suppose $\{\phi _{\alpha }\}$ is a net of completely
positive maps in $K(\phi _1, \phi _2)$ converging to a completely positive map $\phi
$ in BW topology. Then $$\liminf _{\alpha
}~\|\sigma_1-\sigma_2\|^{\phi _{\alpha }}_{cb}\geq
\|\sigma_1-\sigma_2\|^{\phi}_{cb}.$$

\end{lemma}

\begin{proof}
 The following simple observation is used: Let $f(a,b)$
be a real valued function on two variables $a,b.$ Then $\inf_a\sup_b f(a,b)
\geq \sup_b\inf_a f(a,b).$ Now
\begin{eqnarray*} && \liminf_{\alpha
}~\|\sigma_1-\sigma_2\|^{\phi_\alpha}_{cb} \\ &=& \liminf_{\alpha }~
\sup_{\|(a_{ij})\| \leq 1, n, k,
\tilde{c}_i,\tilde{b}_i,\tilde{g}_i,1\leq i\leq n}~ F(\phi_ \alpha ,n,k,
(a_{ij}),(\tilde{c}_i),(\tilde{b}_i),(\tilde{g}_i)) \\ &\geq &
\sup_{\|(a_{ij})\| \leq 1, n, k,
\tilde{c}_i,\tilde{b}_i,\tilde{g}_i,1\leq i\leq
n}~\lim_ \alpha ~F(\phi_ \alpha ,n,k,
(a_{ij}),(\tilde{c}_i),(\tilde{b}_i),(\tilde{g}_i)) \\ &=&
\sup_{\|(a_{ij})\| \leq 1, n, k,
\tilde{c}_i,\tilde{b}_i,\tilde{g}_i,1\leq i\leq n}~F(\phi,n,k,
(a_{ij}),(\tilde{c}_i),(\tilde{b}_i),(\tilde{g}_i)) \\ &=&
\|\sigma_1-\sigma_2\|^{\phi}_{cb}.  \end{eqnarray*}

\end{proof}

\begin{proposition}
 There is a $\phi \in K(\phi_1,\phi_2)$ for which the infimum is attained for $\gamma(\phi_1,\phi_2),$ that is,
  $$ \gamma(\phi_1,\phi_2) = \|\sigma_1-\sigma_2\|^{\phi}_{cb}. $$

 \begin{proof} This follows from the compactness of $K(\phi_1,\phi_2)$ in BW topology and the previous Lemma.
  The definition of $\gamma(\phi_1,\phi_2)$ will give a sequence of unital completely positive maps $\phi_n\in \mbox{UCP}(\phi_1,\phi_2)$ such that $\gamma(\phi_1,\phi_2)=\lim _n~\|\sigma_1-\sigma_2\|^{\phi_n} _{cb} .$ From compactness, we may find a subnet
   $\phi_{\alpha }$ converging to $\phi$ in BW topology. Note that $\gamma(\phi_1,\phi_2)=\lim~\|\sigma_1-\sigma_2\|^{\phi_\alpha } _{cb}.$ From the Lemma we get  $\lim~\|\sigma_1-\sigma_2\|^{\phi_{\alpha }} _{cb} \geq \|\sigma_1-\sigma_2\|^{\phi} _{cb}.$ This implies that $\gamma(\phi_1,\phi_2)=\|\sigma_1-\sigma_2\|^{\phi}_{cb}.$

\end{proof}
\end{proposition}

This result shows that $\gamma $ is attained when the range algebra
$\mathcal B$ is a von Neumann algebra. In view of Theorem
\ref{injective}, it holds good also for injective $C^*$-algebras. In
other words, we have the following result.

\begin{theorem}
Suppose $\mathcal A$ is a $C^*$-algebra and $\mathcal B$ is a von
Neumann algebra or an injective $C^*$-algebra.  Suppose $\phi_1,
\phi_2\in \mbox{UCP}(\mathcal A, \mathcal B).$ Then there is a joint
representation tuple $(\mathcal E, \sigma_1,\sigma_2,x)$ for
$\phi_1$ and $\phi_2$ such that
$$\gamma(\phi_1,\phi_2)=\gamma_\mathcal E(\phi_1,\phi_2)=\|\sigma_1-\sigma_2\|_{cb}.$$
\end{theorem}

\section{Relationship of representation metric with Bures metric}

Suppose $\mathcal A, \mathcal B$ are $C^*$-algebras, $\mathcal B$ is
injective and  suppose $\phi_1,\phi_2 \in \mbox{UCP} (\mathcal A,
\mathcal B)$. Then we wish to show
\begin{equation}
\beta ^2(\phi _1, \phi _2)= 2-\sqrt {4-\gamma ^2(\phi _1, \phi _2)
}.
\end{equation}
Here for notational convenience we write  $\beta ^2(\phi _1, \phi
_2)$ instead of $[\beta (\phi _1, \phi _2)]^2$, with similar
notation for $\gamma .$ First we prove the result for states, that
is, when $\mathcal B = \mathbb C.$
 Actually,  we first prove
(Theorem \ref{MainTheorem}  ):
$$ \gamma(\phi_1,\phi_2) = \beta(\phi_1,\phi_2)\sqrt{4-\beta^2(\phi_1,\phi_2)},$$
and then solve the associated quadratic equation to get  $$\beta ^2
(\phi_1,\phi_2) = 2\pm \sqrt{4-\gamma ^2(\phi_1,\phi_2)}$$ and
observe that only the negative sign is permissible, as $0\leq \beta
(\phi_1,\phi_2)\leq \sqrt{2}$ and $0\leq \gamma (\phi_1,\phi_2) \leq
2$ are trivially true for unital completely positive maps.

It is to be recalled that when we are dealing with states, the GNS
Hilbert $C^*$-modules under consideration  are just Hilbert spaces.
Suppose $(H, \pi , x_1, x_2)$ is a common representation for
$\phi_1$ and $\phi_2.$ We take $S(\pi , \phi)= \{ x: \phi (\cdot )=
\langle x, \pi (\cdot )x\rangle \}.$  To begin with we obtain some
lower and upper bounds of representation metric for {\bf states} on
$C^*$-algebras. In the following, $d(A,B)$ stands for the distance
between sets $A,B$ in relevant metric spaces.
 \begin{lemma}\label{Ad}

 Let $x,y$ be unit vectors in a Hilbert space $K$. For a unitary $U$ in $K$, denote by $Ad_U$ the automorphism $X\mapsto UXU^*$, on $\mathbb B(K)$. Then \begin{eqnarray*}
  \inf_{U:Ux=y}~\|id-Ad_U\|_{cb} &=& \inf_{U:Ux=y}~\|id-Ad_U\| \\ &=& 2
  \inf_{U:Ux=y}~ d(U,\mathbb C) \\ &=&  2 \sqrt{1-|\langle x,y\rangle |^2}. \end{eqnarray*} Moreover, the infimum is attained.

 \end{lemma}

 \begin{proof}

 For any unitary $U$ on $K$, from (Stampfli \cite{Sta}), we see $\|id-Ad_U\|=2d(\mathbb CI,U).$  For
 $n\in \mathbb{N}$, denoting $\tilde{U}=U\oplus\cdots\oplus U$, on $K^n=K\oplus\cdots\oplus K,$ ($n-$ copies),
  we see $ d(\mathbb CI,\tilde{U})=  d(\mathbb CI,U)$ and hence
  $\|id-Ad_U\| _{cb}= 2d(\mathbb CI,U).$   Now if $U$ is a unitary such that $Ux=y$,
  for any $\lambda \in \mathbb{C},$
$$\| (U-\lambda)\| ^2 \geq \| (U-\lambda I)x\| ^2 = \| y-\lambda x\| ^2 \geq 1 -| \langle x, y\rangle |^2,$$
where the last inequality follows as $x,y$ are unit vectors and
$$ |\lambda |^2 -\lambda \langle y, x\rangle -\bar{\lambda }\langle x, y\rangle + |\langle x, y\rangle |^2=
|(\lambda - \langle x, y\rangle)|^2 \geq 0.$$ By considering a
unitary $U$, satisfying $Ux=y$, and $Uv=v$ on $\{ x,y\}^{\perp}$, it
is easily seen that the infimum in $ \inf_{U:Ux=y}~ d(U,\mathbb C)$
is attained and equals $  2 \sqrt{1-|\langle x,y\rangle | ^2}.$

 \end{proof}

As an immediate consequence of this lemma we get the following
bounds for $\gamma$ of pairs of states on $C^*$-algebras.

\begin{theorem}
Let $\phi_1,\phi_2$ be two states on some $C^*$-algebra $\mathcal A.$ Then
 $$ \|\phi_1-\phi_2\| \leq \gamma(\phi_1,\phi_2) \leq 2 \sqrt { \| \phi _1 -\phi _2\|}$$
\end{theorem}

 \begin{proof} Note that for a linear functional the norm and the completely bounded norm coincide.
 The lower bound is now clear from the definition of $\gamma .$
 From Proposition 1.6, \cite{Bur}, we know that there is a common representation space in which Bures distance is
 attained. Let $(K,\pi,x,y)$ be a common representation at which the Bures  distance
 for $(\phi _1, \phi _2)$
  is attained.
Consider $\pi_1=\pi$ and $\pi_2=U^*\pi U,$ where $U$ is a unitary on
$K$ such that $Ux=y.$ Then $(K,\pi_1,\pi_2,x)$ is
a joint representation of $(\phi_1,\phi_2).$ We get \begin{eqnarray*} \gamma(\phi_1,\phi_2) &\leq& \inf_{U:Ux=y}~ \|\pi(\cdot)-U^*\pi(\cdot)U\|_{cb} \\ & \leq & \|id_{\mathcal B(K)}-U^*id_{\mathcal B(K)}U\|_{cb} \\ &=& 2 \sqrt{1-|\langle x,y\rangle|^2} \\ &=& \beta(\phi_1,\phi_2)\sqrt{4-\beta^2(\phi_1,\phi_2)} \\
  &\leq& 2\beta(\phi_1,\phi_2) \\ &\leq & 2 \sqrt{\|\phi_1-\phi_2\|},   \end{eqnarray*}
where the last inequality is from \cite{Bur} and \cite{Con}.
\end{proof}


Now we come to our main theorem on relationship between $\beta $ and
$\gamma .$

\begin{theorem}\label{MainTheorem}

Suppose $\phi_1,\phi_2$ are two states on a $C^*$-algebra $\mathcal A.$ Then
$$\beta ^2(\phi _1, \phi _2)= 2-\sqrt {4-\gamma ^2(\phi _1, \phi _2) }.$$
\end{theorem}


The key to the proof of Theorem \ref{MainTheorem} is the following Lemma.

\begin{lemma}\label{difference}

Let  $x, y$ be unit vectors on a Hilbert space $K$ and let $W$ be a
unitary on $K$ such that $Wx=y$.  Let $P$ be any positive operator
on $K.$ Then $$ \|W-P\| \geq \sqrt{1-[\mbox{Re}\langle
x,y\rangle]^2}.$$

\end{lemma}

\begin{proof} Let $\lambda \in \sigma (W).$ As $W$ is normal for $\epsilon >0$, there exists a
unit vector $v_{\epsilon}\in K$ such that
$$|\langle v_{\epsilon}, Wv_{\epsilon}\rangle -\lambda |<\epsilon .$$
Moreover as $P$ is positive, $\langle v_{\epsilon}, Pv_{\epsilon}\rangle \in \mathbb {R}_+.$ Observe that as $\lambda \in \sigma (W),$
$$d(\lambda , \mathbb{R}_+) =\left\{ \begin{array}{lcl}
1& \mbox{if} & \mbox{Re} (\lambda )\leq 0;\\
\mbox {Im }(\lambda )& \mbox{if} & \mbox{Re} (\lambda )>0.
\end{array} \right.$$
Hence if there exists $\lambda \in \sigma (W)$ with $Re(\lambda
)\leq 0$, we get
$$\|
W-P \| \geq | \langle v_{\epsilon} , (W-P) v_{\epsilon}\rangle | \geq (1-\epsilon), $$
for every $\epsilon >0.$ That is, $\| W-P\| \geq 1.$ Then the result follows trivially as
$\sqrt{1-[\mbox{Re}\langle x,y\rangle]^2}\leq 1.$

So we may assume $\mbox{Re}( \lambda )>  0$ for every $\lambda \in
\sigma (W)$. Now   as $\langle x , y\rangle $ is in the numerical
range of unitary $W$,  it is in the convex hull of $\sigma (W).$
Consequently $\mbox{Re}(\langle x, y\rangle )\geq 0$ and there
exists $\lambda $ in $\sigma (W)$ such that $0\leq \mbox{Re}(
\lambda ) \leq \mbox{Re} (\langle x, y \rangle ),$ or $\mbox {Im} (
\lambda ) \geq \sqrt { 1- \mbox{Re} (\langle x, y \rangle )^2}. $
For $\epsilon >0$, choose $v_{\epsilon}$ as before. Therefore  $\|
W-P\|\geq |\langle v_{\epsilon}, (W-P)v_{\epsilon}\rangle | \geq
d(\lambda , \mathbb {R} _+)-\epsilon = \mbox {Im}~ (\lambda
)-\epsilon. $ As $\epsilon >0$ is arbitrary this completes the
proof.

\end{proof}

We also need the following well-known theorem.
\begin{theorem}\label{Johnson} (Johnson \cite{Joh}) Suppose $\pi$ is a faithful representation
of a $C^*$-algebra $\mathcal A$ on a Hilbert space $K$ and $U$ is a
unitary on $K.$ Then $\|\pi -Ad_U\circ \pi \|_{cb}=2
d(U,\pi(\mathcal A)^\prime).$

\end{theorem}

\begin{proof}

Making use of Kaplansky density theorem, we may replace the
$C^*$-algebra  $ \pi (\mathcal A ) $ by the von Neumann algebra
generated by it. Now the result follows from Theorem 7 of
\cite{Joh}.

\end{proof}

\begin{lemma}\label{attainment}
Suppose $\pi$ is a faithful representation of a $C^*$-algebra
$\mathcal A$ on a Hilbert space $K$ and $U$ is a unitary on $K.$
Then there exists $X\in \pi (\mathcal A )'$ such that $d(U, \pi
(\mathcal A)')=\| U-X\|.$
\end{lemma}
\begin{proof}

 This is an application of the fact that inf-sup is greater than sup-inf.
Indeed, from the definition of infimum, there is a sequence $\{
X_n\}_{n\geq 1} $ in $\pi(\mathcal A)^\prime$ such that
$\|U-X_n\|\leq d(U,\pi(\mathcal A)^\prime)+\frac{1}{n}.$ Observe
that as $I\in \pi (\mathcal A)^\prime ,$ trivially $d(U, \pi
(\mathcal A)^\prime ) \leq 1.$ Consequently  $\|X_n\| \leq \|U-X_n\|
+ \|U\| \leq 1+\frac{1}{n}+1\leq 3$. So $\{ X_n\} _{n\geq 1} $ is a
norm bounded sequence. Hence it has a WOT convergent subnet
converging to some $X$ (say). Clearly $X\in \pi(\mathcal A)^\prime$
as $\pi(\mathcal A)^\prime$ is WOT closed. Now
$$\|U-X\| = \sup_{\|z\|\leq 1,\|w\| \leq 1}|\langle z,
(U-X)w\rangle| . $$ Hence for $\epsilon >0$, there exist $z, w \in
K, \|z\|, \|w\|\leq 1$, such that $\| U-X\| < |\langle z,
(U-X)w\rangle |+\epsilon .$ Then by WOT convergence, we get $n\geq
1$, such that $|\langle z, (X_n-X)w\rangle | < \epsilon $ and $\|
(U-X_n)\|< d(U, \pi (\mathcal A)^\prime )+\epsilon .$ Combining all
three inequalities, we have $\| U-X\| < |\langle z, (U-X_n)w\rangle
|+ |\langle z, (X_n-X)w\rangle |+ \epsilon \leq d(U, \pi (\mathcal
A)^\prime )+3\epsilon .$ As $\epsilon >0$ is arbitrary, we conclude
that $\| U-X\|=d(U, \pi(\mathcal A)^\prime ).$

\end{proof}

\textbf{Proof of Theorem \ref{MainTheorem}} :  Given two
representations $\pi_1,\pi_2$ of $\phi_1,\phi_2$ respectively on
some Hilbert space $K$ together with $x\in K,$ such that $ \phi
_1(\cdot )=\langle x,\pi_1(\cdot)x\rangle $ and $\phi_2(\cdot
)=\langle x,\pi_2(\cdot)x\rangle $,  we may consider unitarily
equivalent representations $\pi_1\oplus \pi_2 $ and $\pi_2\oplus
\pi_1$ on $K\oplus K$ with $x\oplus 0\in K\oplus K.$ This does not
change the norm difference. In other words, we may restrict
ourselves  with unitarily equivalent representations $\pi_1,\pi_2$
on $\mathcal K.$ Suppose $U$ is a unitary on $\mathcal K$ which
intertwine $\pi_1$ and $\pi_2.$ Let $y=Ux.$ So we are led to
consider all tuple $(\pi,K,x,y,U)$ such that $ \phi_1(\cdot)=\langle
x,\pi(\cdot)x\rangle $ and $\phi_2(\cdot)=\langle y,\pi(\cdot)y\rangle,$ $Ux=y.$ It follows that $$ \gamma(\phi_1,\phi_2) =
\inf_{\{\pi,K,U,x,y\}}~ \|\pi-U^*\pi U\|_{cb}.$$ Suppose
$(\pi,K,x,y,U)$ is one such tuple. From Theorem \ref{Johnson}, we
get $$\|\pi-U^*\pi U\|_{cb}= 2 d(U,\pi(\mathcal A)^\prime).$$ Then
by Lemma \ref{attainment}, there exists $X\in \pi (\mathcal A)'$ such that
$\| U-X\|= d(U,\pi(\mathcal A)^\prime).$

Case (i) Every $X$  as above has either non-trivial kernel or has a
range which is not dense (equivalently, $X^*$ has non-trivial
kernel): Clearly in such cases $\|U-X\| = \| U^*-X^*\| \geq 1.$
Suppose in every common representation $\{\pi,K,x,y,U\},$ we find
$X$ with either non-trivial kernel or non-dense range, then we
conclude that $\gamma(\phi_1,\phi_2)=2.$ We shall be done if we show
that in that case $\beta(\phi_1,\phi_2)=\sqrt{2}.$ Indeed in any
common representation $(\pi,K,x,y)$ with $\langle x,y\rangle \neq
0,$ we choose unitary $U$ as in Lemma \ref{Ad}, we see
$\gamma(\phi_1,\phi_2) <2$ contradicting our conclusion. Thus in any
common representation $(\pi,K,x,y),$ we have $\langle x,y \rangle
=0.$ Hence in this case, $\beta (\phi_1,\phi_2)=\sqrt{2}.$

Case (ii) For some tuple $(\pi , K,x,y, U)$, there exists $X$ as
above having trivial kernel and dense range. Taking
polar decomposition of $X=V|X|,$ $V, |X| \in \pi(\mathcal A)^\prime$
with $V$ unitary, we have $\|U-X\|=\|V^*U-|X|\|.$  Now from Lemma
\ref{difference}, we get $$ \|V^*U-|X|\| \geq
\sqrt{1-[\mbox{Re}\langle x,V^*y \rangle]^2}.$$ Note that $V^*y \in
S(\pi, \phi_2).$ Hence $$ \|V^*U-|X|\| \geq \inf_{x^\prime\in
S(\pi,\phi_1), y^\prime\in S(\pi,\phi_2)}\sqrt{1-|\langle
x^\prime,y^\prime \rangle|^2}.$$

 One thing to be noted, while computing Bures distance
for states is that we only need  to consider all common
representations $(K, \pi, x, y)$ such that $\langle
x,y\rangle \geq 0.$ Indeed if $\langle x, y\rangle = |\langle x,
y\rangle |e^{\i\theta}$,   we may change $x$ to $x_1=e^{-i\theta}x.$
Note that $\phi _1(\cdot )= \langle x_1, (\cdot )x_1\rangle $ and
$\|x _1-y\| ^2= 2-2|\langle x, y\rangle | \leq \|x-y\| ^2.$

 It is known that Bures distance is attained. See Lemma 1,\cite{Ara}, Lemma 5.3,
  Proposition 6, \cite{Con} for more details. Consider a common representation $(K, \pi ,x,y)$ in which the Bures distance
 is attained,
 and $\langle x, y\rangle \geq 0$, as $\beta ^2(\phi _1, \phi _2)
 =2-2\langle x, y\rangle, $
  by direct computation,
 $$\sqrt{1-\langle x,y \rangle^2} = \frac{1}{2}\beta(\phi_1,\phi_2)\sqrt{4-\beta^2(\phi_1,\phi_2)}.$$
 We get immediately that for any tuple $(\pi,K,x,y,U)$ as above,  $$\|\pi-U^*\pi U\|_{cb} \geq
 \beta(\phi_1,\phi_2)\sqrt{4-\beta^2(\phi_1,\phi_2)}.$$
 Therefore $$ \gamma(\phi_1,\phi_2) \geq \beta(\phi_1,\phi_2)\sqrt{4-\beta^2(\phi_1,\phi_2)}. $$
  Now for the reverse inequality, choose $(\pi,K,x,y)$ such that Bures
  distance is attained and
  $\langle x, y\rangle \geq 0$. Choose a unitary $U$ with $Ux=y$. Recalling Lemma $\ref{Ad}$,
  we see that \begin{eqnarray*}  \gamma(\phi_1,\phi_2) &\leq& \|\pi-U^*\pi U\|_{cb} \\ &=& 2d(U,\pi(\mathcal A)^\prime) \\ &\leq& 2d(U,\mathbb CI) \\ &=& 2\sqrt{1-\langle x,y \rangle^2} \\ &=& \beta(\phi_1,\phi_2)\sqrt{4-\beta^2(\phi_1,\phi_2)}. \end{eqnarray*} Hence the reverse inequality holds and this proves the theorem.  \qed

Now we extend the main result to injective range algebras. This
requires a non-trivial result of Choi and Li \cite{Choi}. In this
section we would be realizing the Hilbert $C^*$-modules concretely
as operators from one Hilbert space to another and so we will be
using capital letters $X,Y$ etc., to denote elements of the module
and small letters $g,h$ etc., to denote vectors in Hilbert spaces.
If $T$ is a contraction on  a Hilbert space $H$, a unitary $V$ on a
Hilbert space $K\supseteq H$ is said to be a dilation of $T$ if
$T=P_HV|_H.$

\begin{theorem}\label{ChoiLi}
Let $T$ be a contraction on a Hilbert space $H$ satisfying $T+T^*\geq r I$
for some $r\in \mathbb{R}.$ Then there exists a unitary dilation
$V$  of  $~T$ on  $H\oplus H$ satisfying $V+V^*\geq rI.$
\begin{proof} This is Theorem 2.1 of \cite{Choi}, with change of notation being, $A=T, V=U$ and $\mu =-r .$
\end{proof}
\end{theorem}

We also need the following observation about unitary dilations of strict contractions.

\begin{lemma}\label{formofunitary}
    Let $T$ be a strict contraction on a finite dimensional Hilbert space $H.$ Then any
     unitary dilation $V$ of $T$ on $H\oplus H$ is up to unitary equivalence  of the form $$ V= \left(
    \begin{array}{cc}
    T  & -(I-TT^*)^\frac{1}{2}W\\
    (I-T^*T)^\frac{1}{2}  & T^*W\\
    \end{array}
    \right)$$  for some unitary $W$ on $H.$
\end{lemma}

\begin{proof}

    Set $D _1=(I-T^*T)^\frac{1}{2}$ and $D _2=(I-TT^*)^\frac{1}{2}.$  As $T$ is
    a strict contraction $D_1, D_2$ are invertible. Let $$V=\left(
    \begin{array}{cc}
    T  & -T_{12}\\
    T_{21}  & T_{22}\\
    \end{array}
    \right)$$ be any unitary dilation of $T$ on $H\oplus H.$

    From the equation $V^*V=I=VV^*,$ we get $|T_{21}|=D _1$ and $|T^*_{12}|=D _2.$
Therefore from the polar decompositions of $T_{12}$ and $T_{21},$ we
get $T_{21}=U_1D _1$ and $T^*_{12}=U^*_2D _2$ for some unitaries
$U_1$ and $U_2.$ Equating $(1,2)$ entry of $VV^*$ to 0,  we get $TD
_1U^*_1=D _2U_2T^*_{22}.$ Note that $TD _1=D _2T.$ Therefore we get
$TU_1^*=U_2T^*_{22}.$ Hence $T_{22}=U_1T^*U_2.$ Now by direct
calculation we get that $$V=\left(
    \begin{array}{cc}
    I  & 0\\
    0  & U_1\\
    \end{array}
    \right)\left(
    \begin{array}{cc}
    T  & -D _2W\\
    D _1  & T^*W\\
    \end{array}
    \right)\left(
    \begin{array}{cc}
    I  & 0\\
    0  & U^*_1\\
    \end{array}
    \right)$$ where $W=U_2U_1.$

\end{proof}

\begin{lemma}\label{distance}
Let $G,H$ be two Hilbert spaces and    let $X,Y: G\to H$ be two
isometries with $\|X^*Y\|<1.$ Then identifying $H$ with $H\oplus 0,$
there is a unitary $U\in \mathbb B(H\oplus H)$ such that $UX=Y$ and
$$ d(U,\mathbb C) = \sup_{\|g\|=1}\sqrt{1-|\langle
Xg,Yg\rangle|^2}.$$
\end{lemma}

\begin{proof}

    It follows from Lemma \ref{Ad}, that   any unitray $U$ with $UX=Y$ will satisfy $$ d(U,\mathbb C)
\geq \sup_{\|g\|=1}\sqrt{1-|\langle Xg,Yg\rangle|^2}.$$ Set
$T=X^*Y.$ Let $ \Delta :=\overline{W(T)}$ be the closure of the
numerical range of $T.$ If $0\in \Delta ,$ then
$d(U,\mathbb C)=1$ and the lemma follows easily. Therefore assume
$0\notin \Delta .$
    Note that $\Delta $ is a compact convex non-empty subset of $\mathbb C.$ Let $\lambda =re^{i\theta }$ be
    the unique point in $\Delta $ such that $|\lambda | = \inf \{|z|: z\in \Delta \}$. Replacing $Y$ by $e^{-i\theta  }Y,$
    we may assume without loss of generality, $\lambda=r>0.$
    Observe that $$ \sqrt{1-r^2}=\sup_{\|g\|=1}\sqrt{1-|\langle Xg,Yg\rangle|^2}.$$ If $r=1$
    then $\langle g, X^*Yg\rangle = \langle g, g\rangle$ for every $g\in G$ and then $X=Y$. In such a
     case we may take $U=I$, and we are
 done.  Therefore assume $0<r<1.$ Consider
    the vertical line $L_r:=\{ z\in \mathbb C :~~\mbox{ Re }~ ( z)  =r\} $ in the complex plane. Note that this line is
    tangent to the circle centered at $(0,0)$ and radius $r.$ Therefore the
    convexity of $\Delta $ would implies that $\Delta $ can not have any point
    to the left of this line. Therefore $T+T^*\geq 2r.$

    \textbf{Case 1: $G$ is finite dimensional.}
    By Theorem \ref{ChoiLi} and Lemma \ref{formofunitary} there is a unitary $V$ on $G\oplus G$ of the form
    $$ V=\left(
    \begin{array}{cc}
    T  & -(I-TT^*)^{\frac{1}{2}}W\\
    (I-T^*T)^{\frac{1}{2}}  & T^*W\\
    \end{array}
    \right),
    $$ with $W\in \mathbb B(G)$ chosen such a way, we get $V+V^*\geq 2r.$ As $\|T\|<1,$ the operator
     $D_1=(I-T^*T)^{\frac{1}{2}}$ is invertible.
    Define $C=(Y-XT)D_1^{-1}.$  We see that $C^*C=I$ and $X^*C=0.$ In particular, range of $X$ and range of $C$
    are orthogonal.   Taking $H_0:= (X(G)\oplus C(G))^{\perp}, $ decompose $H$ as
    $H=X(G)\oplus C(G)\oplus H_0.$ Define $\tilde{U}|_{H_0}=I$ and on ${H_0}^\bot,$ via the following
    unitary  $$ \tilde{U}_{H_0^{\perp}}= \left(
    \begin{array}{cc}
    X & 0\\
    0 & C\\
    \end{array}
    \right)V \left(
    \begin{array}{cc}
    X^*& 0\\
    0  & C^*\\
    \end{array}
    \right).$$
     Set $U$ on $H\oplus H$ by $U=\tilde{U}\oplus I.$ We see that $\tilde{U}X=Y,$ $UX=Y$ and $\tilde{U}\oplus
      \tilde{U}^*\geq 2r,$ $U+U^*\geq 2r.$ This implies $\sigma(U)$ is to
the right side of the line $L_r$.  As $\sqrt{1-r^2}>1-r,$ we
    see that the circle centred at $(r,0)$ and radius $\sqrt{1-r^2}$ covers $\sigma(U).$
    Therefore $d(U,\mathbb C)\leq \sqrt{1-r^2}.$

    \textbf{Case 2: $G$ is an arbitrary Hilbert space.}

    Let $F$ be any finite dimensional subspace of $G.$ Consider $X_F=X|_F,$ $Y_F=Y|_F,$ $T_F=X^*_FY_F.$ Note
that $T_F+T^*_F\geq 2r.$ By the finite dimensional result, there is
a unitary $U_F\in \mathbb B(H)$ such that $U_FX_F=Y_F$ and
$U_F+U^*_F\geq 2r.$ Since the set of all finite dimensional
subspaces is a directed set under inclusion, the bounded net
$\{U_F\}$ has a WOT convergent subnet to a contraction $\tilde{U}.$
Note that $\tilde{U}+\tilde{U}^*\geq 2r.$  For any $F_1\supset F,$
we get $U_{F_1}Xg=U_{F_1}X_{F_1}g=Y_{F_1}g=Yg.$ Since, the set of
all finite dimensional subspaces containing $F$ is cofinal, we get
$\tilde{U}X=Y.$ Applying Theorem \ref{ChoiLi}, we get a unitary
dilation $U$ of $\tilde{U}$ on $H\oplus H$ with $U+U^*\geq 2r.$ As
$X$ is an isometry, $UX=Y.$ Further as $U+U^*\geq 2r,$ we get
$d(U,\mathbb C)\leq \sqrt{1-r^2}.$
\end{proof}

\begin{lemma}\label{subfamily}
    Let $\mathcal A,\mathcal B$ be $C^*$ algebras. Let $\phi _1, \phi_2\in \mbox{UCP}(\mathcal A,\mathcal B).$ Then $$\beta(\phi _1,\phi _2)=\inf_{\{(\mathcal E,X,Y) ~:~
      \|\langle X,Y\rangle\|<1\}}\|X-Y\|$$ where $(\mathcal E,X,Y)$ is a common representation module
      for $\phi _1,\phi _2$ with $\|\langle X,Y\rangle\| <1.$
\end{lemma}

\begin{proof}

Let $(\mathcal E,X,Y)$ be a common
    representation module for $(\phi _1,\phi _2)$.  For $0<r<1$, take $\tilde{\mathcal E}=
    \mathcal E\oplus \mathcal E,$ $X_r=X\oplus 0,$ $Y_r=rY\oplus \sqrt{1-r^2}Y.$ Then
    $(\tilde{\mathcal E},X_r,Y_r)$ is a common representation module for $\phi _1,\phi _2.$
Further, $\| \langle X_r, Y_r\rangle \| = $ $ r\| \langle X,Y\rangle \|\leq r <1.$  Also
$\lim _{r\to 1} \|  X_r-Y_r\| = \| X-Y\|.$ Hence,
$$\inf_{\{(\mathcal E,X,Y) ~:~
      \|\langle X,Y\rangle\|<1\}}\|X-Y\| = \inf_{\{(\mathcal E,X,Y)\}}\|X-Y\| = \beta (\phi _1, \phi _2).$$

\end{proof}

\begin{lemma}\label{real is mod}
    Let $\mathcal A,\mathcal B$ be unital $C^*$ algebras. Let $\mathcal B\subseteq \mathbb B(G).$
    Let $\phi _1,\phi _2 \in \mbox{UCP}(\mathcal A,\mathcal B).$
     Then $$\beta (\phi _1,\phi _2)=\inf_{(\mathcal E,X,Y)} \sup_{g\in G,\|g\|=1}\sqrt{2}\sqrt{1-|\langle g,
     \langle X,Y\rangle g \rangle|} $$ where $(\mathcal E,X,Y)$ is a common representation module
     for $\phi _1,\phi _2.$
\end{lemma}

\begin{proof}
    Let $(\mathcal E,X,Y)$ be as above and define $$\beta '(\phi _1, \phi _2)=
    \inf_{(\mathcal E,X,Y)} \sup_{\|m\|=1}\sqrt{2}\sqrt{1-|\langle m,
     \langle X,Y\rangle m \rangle|}. $$
     As $\phi _1, \phi _2$ are unital, $\langle X,X \rangle=1=\langle Y,Y\rangle$ and hence,
     $\langle (X-Y), (X-Y)\rangle =2(1-\mbox{Re}(\langle X,Y\rangle ) ).$ So
    \begin{eqnarray*}
    \| X-Y \| ^2 &= & \sup _{g=1} \langle g, \langle (X-Y),(X-Y)\rangle g\rangle \\
    &=& \sup _{g=1} \langle g, 2(1-\mbox{Re}~\langle X,Y\rangle)g\rangle \\
   &=& \sup_{g=1} 2 (1- \langle g, \mbox{Re}~\langle X,Y \rangle )g\rangle \\
   &\geq & \sup_{g=1} 2(1-|\langle g, \langle X,Y \rangle g\rangle |)\\
   \end{eqnarray*}
    Consequently  $\beta (\phi _1,\phi _2)\geq \beta ' (\phi _1,\phi _2) .$
    Suppose the equality does not hold,  then there is a positive number $0<t<\sqrt{2}$ such that $\beta (\phi _1,\phi _2)>t>
     \beta ' (\phi _1,\phi _2).$  We will arrive at a contradiction.
     Let $(\mathcal E,X,Y)$ be a common representation module for $(\phi _1,\phi _2)$ such that
     $$ \sup_{\|g\|=1}\sqrt{2}\sqrt{1-|\langle g, \langle X,Y\rangle g \rangle|} < t.$$ Set
      $T=\langle X,Y \rangle.$ Let $\Delta :=\overline{W(T)}$ be the closure of numerical range of the
      operator $T.$ Note that $\Delta $ is a compact convex non-empty subset of $\mathbb C.$ Note that
       as $t <\sqrt{2},$ we have $0\notin \Delta .$
       Let $\lambda =re^{i\theta }$ be
       the unique point in $\Delta $ such that $|\lambda | = \inf \{|z|: z\in \Delta \}$. Set $\tilde{X}=Xe^{i\theta}.$ Then $(\mathcal E,\tilde{X},Y)$ is
         a common representation module for $(\phi _1,\phi _2).$ The convexity of $W(T)$
         implies \begin{eqnarray*} \|\tilde{X}-Y\| &=& \sup_{\|g\|=1}\sqrt{2}\sqrt{1-\mbox{Re}(\langle g, \langle
         \tilde{X},
         Y\rangle g \rangle)}\\ &=& \sup_{\|g\|=1}\sqrt{2}\sqrt{1-|\langle g, \langle \tilde{X},Y\rangle
          g \rangle|} \\ &=& \sup_{\|g\|=1}\sqrt{2}\sqrt{1-|\langle g, \langle X,Y\rangle g \rangle|}.
           \end{eqnarray*} Therefore $\|\tilde{X}-Y\|<t.$ This implies $\beta (\phi _1,\phi _2)<t.$
           This is a contradiction.
\end{proof}

\begin{theorem} \label{upperbound}
        Suppose $\mathcal A$ is a unital $C^*$-algebra and $\mathcal B\subset \mathbb B(G)$ is an injective $C^*$-algebra. Suppose $\phi_1,\phi_2 \in \mbox{UCP}(\mathcal A, B).$ Then $$ \gamma(\phi_1,\phi_2)= \beta(\phi_1,\phi_2)\sqrt{4-\beta^2(\phi_1,\phi_2)}.$$
\end{theorem}

\begin{proof}
    Let $\tilde{\phi_i}=\iota\circ\phi_i,$ $\iota:\mathcal B\rightarrow \mathbb B(G)$ inclusion map,
     $i=1,2.$ As $\mathcal B$ is injective, we get from Proposition 9, \cite{Con} that $\beta(\tilde{\phi}_1,\tilde{\phi}_2) =
     \beta(\phi_1,\phi_2)$ and from Theorem \ref{injective}, $\gamma(\tilde{\phi}_1,\tilde{\phi}_2) =
     \gamma(\phi_1,\phi_2).$ Therefore we may assume without loss of generality that $\mathcal B=
     \mathbb B(G).$

    We get from \ref{amplifiation} and \ref{composition}, that $$\gamma(\phi_1,\phi_2)\geq
    \gamma(\omega\circ(\phi_1\otimes \mbox{id}_{\mathbb B(G)}),\omega\circ(\phi_2\otimes \mbox{id}_{\mathbb B(G)})),$$
    for every $\omega\in G\otimes G,$ $\|\omega\|=1.$ Denoting $\psi_i=\phi_i\otimes \mbox{id}_{\mathbb B(G)},$
     $i=1,2,$ we get immediately from Theorem  \ref{MainTheorem}, $$ \gamma (\phi _1,\phi _2)\geq
     \sup_{\omega\in
      G\otimes G,\|\omega\|=1}\beta(\omega\circ\psi_1,\omega\circ\psi_2)\sqrt{4-\beta^2(
      \omega\circ\psi_1,\omega\circ\psi_2)}.$$ Note that from Proposition 6, \cite{Con}, we get
     $$ \sup_{\omega\in G\otimes G,\|\omega\|=1}\beta(\omega\circ\psi_1,\omega\circ\psi_2) = \beta (\phi_1,\phi_2).$$ Therefore $$\gamma(\phi_1,\phi_2) \geq \beta(\phi_1,\phi_2)\sqrt{4-\beta^2(\phi_1,\phi_2)}.$$ Let us prove the reverse inequality.
      Let $(\tilde{\pi},H, X,Y)$ be a common representation for $\phi_1,\phi_2$ satisfying $\|X^*Y\|<1.$ I.e. $\tilde{\pi}:\mathcal A\rightarrow
     \mathbb B(H)$ is a representation, $X,Y:G\rightarrow H$ isometries with $\phi_1(\cdot)=X^*\pi(\cdot)X$ and $\phi_2(\cdot)=Y^*\pi(\cdot)Y.$ Set
    $K=H\oplus H.$ Identify $H$ with $H\oplus 0.$ Set $\pi:\mathcal A\rightarrow \mathbb B(K)$ by
    $\pi=\tilde{\pi}\oplus \mbox{id}.$ Then $(\pi,K, X,Y)$ is a common representation for $\phi_1,\phi_2.$ Let
    $U\in \mathbb B(K)$ be a unitary as in Lemma \ref{distance}. Therefore $(K,\pi,U^*\pi U,X)$ is a joint
    representation module for $(\phi_1,\phi_2).$ Now from Theorem \ref{Johnson} and Lemma \ref{distance}, we get
    \begin{eqnarray*}
    \gamma(\phi_1,\phi_2)&\leq& \|\pi-U^*\pi U\|_{cb} \\ &=& 2 d(U,\pi(\mathcal A)^\prime) \\
    &\leq& 2 d(U,\mathbb C) \\ &=& 2\sup_{\|g\|=1}\sqrt{1-|\langle Xg,Yg\rangle|^2}.
    \end{eqnarray*}

     Set $s=\sup_{\|g\|=1}\sqrt{2}\sqrt{1-|\langle g, \langle X,Y\rangle g \rangle|}.$ Now observe that
      $$ s\sqrt{4-s^2}=2\sup_{\|g\|=1}\sqrt{1-|\langle Xg,Yg\rangle|^2}.$$
Observe  that $x\mapsto x\sqrt{4-x^2},$ is an increasing function on
the interval $[0,\sqrt{2}],$      Now as
     $(\tilde{\pi},H,X,Y)$ is an arbitrary common representation space for $\phi_1,\phi_2$
     satisfying $\|X^*Y\|<1$, from Lemma \ref{subfamily} and Lemma \ref{real is mod},
     we conclude $\gamma(\phi_1,\phi_2)\leq \beta(\phi_1,\phi_2)\sqrt{4-\beta^2(\phi_1,\phi_2)}.$

\end{proof}

\section{Examples}

In this section we explore the dependence of the representation
metric $\gamma $ on the range algebra. We see that, when the range
algebra is not injective the relationship between $\beta $ and
$\gamma $ may fail. The examples draw upon ideas from \cite{BhS}.

\begin{example}
Let  $H$ be a separable  infinite dimensional Hilbert space. Let
$\mathcal K$ denote the set of all compact operators on $H.$ Set
$\mathcal B = \mathcal K_+:=\mbox{span}\{\mathcal K, \mathbb C I_H
\}$,  the unital $C^*$-algebra generated by compact operators. Let
$u\in \mathbb B(H)$ be a unitary of the form $u= \lambda
p+\bar{\lambda}(1-p)$, where $p$ is a projection such that $p$ and
$(1-p)$ have infinite rank and $\lambda = e^{i\theta }$ for some
$0<\theta < \frac{\pi}{2}.$ Clearly $u$ is not in $\mathcal B. $
Define unital $*$-automorphisms $\psi_1,\psi_2$ of $\mathcal B$ by
$\psi_1(a)=u^*au, \psi_2(a)=a.$ Let $\iota : \mathcal B \to {\mathbb
B}(H)$ be the inclusion map and let $\tilde {\psi }_j = \iota \circ
\psi _j$ for $j=1,2.$

Then from Example 3.2 of  \cite{BhS}, we get $\beta(\psi_1,\psi_2)=\sqrt{2}.$ Now
observe that $\mathcal B$ is a $\mathcal B$ right-module with
natural action and define adjointable left actions
$\sigma_1(a)=u^*au$ and $\sigma_2(a)=a.$ Now note that $(\mathcal
E,1,\sigma_1,\sigma_2)$ is a joint representation module for
$\psi_1,\psi_2.$ Therefore
$$\gamma(\psi_1,\psi_2) =
\|\sigma_1-\sigma_2\|_{cb}=\|\psi_1-\psi_2\|_{cb}= 2d(u,\mathbb
C).$$ Therefore $$\gamma(\psi_1,\psi_2)=|\lambda-\bar{\lambda}| <2 =
\beta(\psi_1,\psi_2)\sqrt{4-\beta^2(\psi_1,\psi_2)}.$$ On the other
hand, as $\beta(\tilde{\psi}_1,\tilde{\psi}_2)=
\sqrt{2}(1-\mbox{Re}\lambda)^{\frac{1}{2}}.$ We get
$$\gamma (\tilde{\psi }_1, \tilde {\psi }_2)= \beta(\tilde{\psi}_1,\tilde{\psi}_2)\sqrt{4-\beta^2(\tilde{\psi}_1,\tilde{\psi}_2)}=
2\sqrt{1-(\mbox{Re}\lambda)^2}=2d(U,\mathbb
C)=\gamma(\psi_1,\psi_2).$$  \end{example}

In the previous example, the representation metric did not change by the
inclusion map. However, it is not always the case. To see this, we
need a slightly more delicate example.

\begin{example} Let $H$ be a separable infinite dimensional Hilbert
space. Let $\mathcal B=\mathcal K _+$ be as in previous Example. Let
$\mathcal A\subset \mathbb B(H\oplus H)$ be the $C^*-$algebra:
$$\mathcal A=\{ \left(
\begin{array}{cc}
X  & Y\\
Z  & W\\
\end{array}
\right) : X,W \in \mathcal K _+, Y,Z \in \mathcal K\} .$$ Let $p$ be
a projection on $H$ such that range of $p$ and $1-p$ are both
infinite dimensional subspaces of $H.$

Let $0<\theta<\frac{\pi}{4}.$ Set $$u :=
e^{i\theta}p+e^{-i\theta}(1-p).$$ Then $u$ is a unitary and  $u\notin
\mathcal K _+.$ Let $$z_1=\frac{1}{\sqrt{2}}\left(
\begin{array}{c}
I\\
I\\
\end{array}
\right),~~z_2=\frac{1}{\sqrt{2}}\left(
\begin{array}{c}
u\\
I\\
\end{array}
\right).$$ Define unital CP maps $\phi_i:\mathcal A\rightarrow
\mathcal B,$ by $\phi_i(a)=z^*_iaz_i, a\in \mathcal A, i=1,2.$

Let $\iota:\mathcal B\rightarrow \mathbb B(H)$ be the inclusion
map. Let $\tilde{\phi_i}=\iota\circ \phi_i, i=1,2.$ As $\mathbb
B(H)$ is injective, we have
$$\gamma(\tilde{\phi_1},\tilde{\phi_2})=\beta(\tilde{\phi}_1,\tilde{\phi}_2)\sqrt{4-\beta^2(\tilde{\phi}_1,\tilde{\phi}_2)}.$$
Set $G=H\oplus H.$ To compute
$\beta(\tilde{\phi}_1,\tilde{\phi}_2),$ first note that $(G,id,z_i)$
is a Stinespring representation for $\tilde{\phi}_i, i=1,2.$ Any
operator $W\in \mathbb B(G)$ in commutant of the identity
representation of compact operators is of the form $W=\lambda I$
with $\lambda \in \mathbb C.$ Therefore
\begin{eqnarray*} \beta(\tilde{\phi}_1,\tilde{\phi}_2) &=&
\inf_{|\lambda|\leq 1} \|(z_1\oplus 0)-(\lambda z_2 \oplus
\sqrt{1-|\lambda|^2}z_2)\|^{\frac{1}{2}} \\ &=&  \inf_{|\lambda|\leq
1} \|2I-2 \mbox{Re} (\lambda z^*_1z_2)\|^{\frac{1}{2}} \\ &=&
\sqrt{2} \|I- \mbox{Re}(\lambda \frac{u+I}{2})\|^{\frac{1}{2}}.
\end{eqnarray*} After simple calculation, we observe that the
infimum is attained at $\lambda=1.$ Therefore
\begin{eqnarray*}\beta(\tilde{\phi}_1,\tilde{\phi}_2) &=&
\sqrt{2}\|I-\mbox{Re} (\frac{u+I}{2})\|^\frac{1}{2} \\ &=&
\|I-\mbox{Re} (u)\|^{\frac{1}{2}} \\ &=& \sqrt{1-\cos\theta}.
\end{eqnarray*} So
$$\gamma(\tilde{\phi},\tilde{\phi}_2)=\sqrt{3-2\cos\theta-\cos^2\theta}=\sqrt{(3+\cos\theta)(1-\cos\theta)}.$$

Let us now compute $\gamma(\phi_1,\phi_2).$ For that we consider
common representation modules $(\mathcal F,x_1,x_2)$ with unitary $U$, where
$\mathcal F$ is an $\mathcal A-\mathcal B$ bi-module, $x_i\in
S(\mathcal F,\phi_i), i=1,2$ and $U\in \mathcal B^a(\mathcal F)$
satisfies $Ux_1=x_2.$

Take ${ \mathcal K } _u= \mbox{span}~\{ \mathcal K , u\} $ and  $$ \mathcal E_1=\left(
\begin{array}{c}
\mathcal K_+\\
\mathcal K_+\\
\end{array}
\right), \mathcal E_2=\left(
\begin{array}{c}
{\mathcal K }_u\\
\mathcal K_+\\
\end{array}
\right) . $$ Then $\mathcal E_1,\mathcal E_2\subset \mathbb B(H,G)$
are $\mathcal A-\mathcal B$ bi-modules. Set $$\mathcal E:=\mathcal
E_1\oplus \mathcal E_2\subset \mathbb B(H,G\oplus G).$$ \end{example}

\begin{lemma}\label{identification}
Suppose $(\mathcal F,x_1,x_2)$ is a common representation module for
$\phi_1$ and $\phi_2.$ Then there is an $\mathcal A-\mathcal B$
bi-module $\mathcal G$ and a bilinear unitary $W\in \mathcal
B^a_{bil}(\mathcal F,\mathcal E\oplus \mathcal G)$ such that
$Wx_1=(z_1,0,0)$ and $Wx_2=(0,z_2,0).$
\end{lemma}

\begin{proof}
Set $M=\mathcal F\odot H.$ Then $\mathcal F\subset \mathbb B(H,M).$ Define $\pi:\mathcal A\rightarrow \mathbb B(M)$ by $\pi(a)(e\odot h)=ae\odot h.$ Set $K_i=\overline{\mbox{span}}\{ax_i\odot h: a\in \mathcal A, h\in H \}.$ Then $K_i$ is a reducing subspace for $\pi.$ Define unitary $U_i:K_i\rightarrow G$ by $U_i(ax_i\odot h)=az_ih,$ $a\in \mathcal A, h\in H, i=1,2.$ Note that $\overline{\mbox{span}}\{az_ih:a\in \mathcal A,h\in H\}=G.$

Identifying $K_1$ with $G$ via unitary $U_1,$ we get $M=G\oplus G^\bot$ and for some representation $\pi _0$, $$\pi=\left(
\begin{array}{cc}
id  & 0\\
0  & \pi _0\\
\end{array}
\right),~x_1=\left(
\begin{array}{c}
z_1\\
0\\
\end{array}
\right),~ x_2=\left(
\begin{array}{c}
w\\
v\\
\end{array}
\right).$$

Now $z^*_2az_2=\phi_2(a)=\langle x_2,ax_2\rangle=w^*aw+v^*\pi^\bot(a)v.$    So $\psi:\mathcal A\rightarrow \mathbb B(G),$ defined by $\psi(a)=w^*aw$ is a CP map dominated by $\phi_2.$ It follows that $w=cz_2$ for some $c\in \mathbb C$, with $ |c|\leq 1.$ Now $\langle x_1,x_2\rangle=z^*_1w=cz^*_1z_2.$ By direct computation, $cz^*_1z_2=\frac{u+I}{2}\notin \mathcal A.$ Therefore $c=0.$ Hence $w=0$ and $\langle x_1,x_2\rangle=0.$ Also by direct computation, we get $\langle \pi(\mathcal A)x_1, \pi(\mathcal A)x_2\rangle=0.$

Similarly identifying $K_2$ with $G$ via unitary $U_2,$ we get $M=G\oplus G\oplus L$ and with some representation $\pi _1$,  $$ \pi=\left(
\begin{array}{ccc}
id  & 0 & 0\\
0   & id & 0\\
0  & 0 & \pi _1\\
\end{array}
\right),~x_1=\left(
\begin{array}{c}
z_1\\
0\\
0\\
\end{array}
\right),~ x_2=\left(
\begin{array}{c}
0\\
z_2\\
0\\
\end{array}
\right).$$

It follows that,  $\mathcal E\cap \mathbb B(H,G\oplus 0\oplus 0)=\mathcal E_1\oplus 0\oplus 0$ and $\mathcal E\cap \mathbb B(H,0\oplus G\oplus 0)=0\oplus \mathcal E_2\oplus 0.$
Consequently $$ \mathcal E = \mathcal E_1\oplus \mathcal E_2 \oplus \mathcal E_3,$$ where $\mathcal E_3=\mathcal E\cap \mathbb B(H,0\oplus 0\oplus L).$

\end{proof}

In view of Lemma \ref{identification}, we consider common representations of the form
$(\mathcal E\oplus \mathcal G, (z_1,0,0),(0,z_2,0))$ with unitary
$U\in \mathcal B^a(\mathcal E\oplus \mathcal G)$ with
$U(z_1,0,0)=(0,z_2,0).$

Let $P\in \mathcal B^a(\mathcal E\oplus \mathcal G)$ be the
projection onto $\mathcal E.$ Set $$V=PUP|_\mathcal E.$$ Then $V\in
\mathcal B^a(\mathcal E)$ is a contraction with $V(z_1,0)=(0,z_2).$
Let $\sigma:\mathcal A\rightarrow \mathcal B^a(\mathcal E\oplus
\mathcal G)$ be the left action. Note that $$\sigma=id\oplus id
\oplus \sigma_\mathcal G.$$ Therefore observe that
$$d(U,\sigma(\mathcal A)^\prime)\geq d(V,(id\oplus id)^\prime).$$

\begin{lemma}\label{adjointable}
    Let $\mathcal E_1,\mathcal E_2$ be as defined earlier. Then $$
    \mathcal B^a(\mathcal E_1,\mathcal E_2)=\left(
    \begin{array}{cc}
    {\mathcal K} _u & {\mathcal K} _u\\
    \mathcal K_+  & \mathcal K_+\\
    \end{array}
    \right)\subset \mathbb B(G).$$
\end{lemma}

\begin{proof}
    We observe that an operator $X\in \mathbb B(G)$ is in $ \mathcal B^a(\mathcal E_1,\mathcal E_2)$ if and only if $XE\in \mathcal E_2$ for every $E\in \mathcal E_1.$ As $$\left(
    \begin{array}{c}
    I\\
    0\\
    \end{array}
    \right), \left(
    \begin{array}{c}
    0\\
    I\\
    \end{array}
    \right)\in \mathcal E_1,$$ we get the result by direct computation.
\end{proof}

Decomposing $$\mathcal B^a(\mathcal E)=\left(
\begin{array}{cc}
\mathcal B^a(\mathcal E_1)  & \mathcal B^a(\mathcal E_2,\mathcal E_1)\\
\mathcal B^a(\mathcal E_1,\mathcal E_2)  & \mathcal B^a(\mathcal E_2)\\
\end{array}
\right),$$ and the fact that $V(z_1,0)=(0,z_2),$ $V^*(0,z_2)=(z_1,0),$ We get $$V=\left(
\begin{array}{cc}
*  & *\\
Z  & *\\
\end{array}
\right),$$  for some $Z\in \mathcal B^a(\mathcal E_1,\mathcal E_2)$ satisfying $Zz_1=z_2, Z^*z_2=z_1.$ Recalling the
choice of $z_1, z_2$, and the definitions of $\mathcal E_1,\mathcal E_2,$ we observe from Lemma \ref{adjointable}, $$Z=\left(
\begin{array}{cc}
au+k  & (1-a)u-k\\
(1-a)I-u^*k  & aI+u^*k\\
\end{array}
\right),$$ for some $k\in \mathcal K$ and $a\in \mathbb C.$ Now \begin{eqnarray*}
\|\sigma-U^*\sigma U\|_{cb} &=& 2d(U,\sigma(\mathcal A)^\prime)\\ &\geq& 2d(V,(id\oplus id)^\prime) \\ &=& 2d(\left(
\begin{array}{cc}
    *  & *\\
    Z  & *\\
\end{array}
\right),\left(
\begin{array}{cc}
\mathbb C I & \mathbb C I\\
\mathbb C I  & \mathbb C I\\
\end{array}
\right)) \\ &\geq& 2d(Z,\mathbb C).
\end{eqnarray*}

 As we have started with arbitrary common representation module, we get $$\gamma(\phi_1,\phi_2) \geq 2d(Z,\mathbb C).$$

 Set $$Z^\prime=\left(
 \begin{array}{cc}
 au  & (1-a)u\\
 (1-a)I  & aI+u\\
 \end{array}
 \right).$$ Therefore $Z=Z^\prime+K$ for some compact operator $K$ on $H\oplus H.$
Let $\{e_n\}$ be a sequence of orthonormal vectors such that
$e_{2n}\in \mbox{range}~p$ and $e_{2n+1}\in \mbox{range}~(1-p).$ Set
$f_n=e_n\oplus 0.$
 Given $\epsilon>0$ find $N$ such that $n\geq N,$ $\|Kf_n\|<\epsilon.$ Now \begin{eqnarray*}\|Z-\lambda I\| &\geq&  \|(Z^\prime -\lambda I + K)f_n\| \\ &\geq& \|(Z^\prime-\lambda I)f_n\| - \epsilon \\ &=& \|\left(
    \begin{array}{c}
        (au-\lambda)e_n\\
        (1-a)e_n\\
    \end{array}
    \right)\|-\epsilon \\ &=& [ A(\lambda)^2 + |1-a|^2]^\frac{1}{2}-\epsilon
  \end{eqnarray*}
where $A(\lambda)=\mbox{max} \{ |ae^{i\theta}-\lambda|,
|ae^{-i\theta}-\lambda|\}.$ Note that infimum of $A(\lambda)$ is
attained at $\lambda=a\cos\theta.$  Therefore taking limit
$\epsilon\downarrow 0,$ we get \begin{eqnarray*}
  d(Z,\mathbb C) &\geq& [|a|^2\sin^2\theta + |1-a|^2]^\frac{1}{2}\\ &\geq& [|a|^2\sin^2\theta + (1-|a|)^2]^\frac{1}{2}
  \end{eqnarray*}
Consider the quadratic polynomial $f:\mathbb R\rightarrow \mathbb R$
defined by $f(r)=r^2\sin^2\theta+(1-r)^2 $ Then $$f(r)= \frac{\sin
^2\theta}{1+\sin ^2\theta}+ (1+\sin ^2\theta )(r-\frac{1}{1+\sin
^2\theta})^2\geq \frac{\sin ^2\theta}{1+\sin ^2\theta}.$$
 Therefore $$ \gamma(\phi_1,\phi_2)\geq 2 d(Z,\mathbb C) \geq \frac{2\sin\theta}{\sqrt{1+\sin^2\theta}}.$$

 Note that  for $0<\theta<\frac{\pi}{4},$ $\sqrt{1+\sin^2\theta}< \sqrt{1+\cos^2\theta}<\sqrt{1+\cos\theta}.$

Therefore \begin{eqnarray*} \gamma(\phi_1,\phi_2)&\geq&\frac{2\sin\theta}{\sqrt{1+\sin^2\theta}} \\ &>& \frac{2\sin\theta}{\sqrt{1+\cos\theta}} \\
&=& 2\sqrt{1-\cos\theta} \\ &>& \sqrt{(3+\cos\theta)(1-\cos\theta)}
\\&=& \gamma(\tilde{\phi}_1,\tilde{\phi}_2). \end{eqnarray*} In
particular, $\gamma(\phi_1,\phi_2)\neq
\gamma(\tilde{\phi}_1,\tilde{\phi}_2).$

Now consider the following common representation $(\mathcal E,
(z_1,0),(0,z_2), V)$ where $V\in \mathcal B^a(\mathcal E)$ is given
by $$ V= \left(
\begin{array}{cc}
0 & 0\\
Z & 0\\
\end{array}
\right),$$ with $$Z=\left(
\begin{array}{cc}
u & 0\\
0  & I\\
\end{array}
\right).$$ We see that $W(z_1,0)=(0,z_2)$ and   \begin{eqnarray*}
    \|\sigma-V^*\sigma V\|_{cb} &=& 2d(V,(id\oplus id)^\prime) \\ &=& 2d(\left(
    \begin{array}{cc}
        0  & 0\\
        Z  & 0\\
    \end{array}
    \right),\left(
    \begin{array}{cc}
        \mathbb C I & \mathbb C I\\
        \mathbb C I  & \mathbb C I\\
    \end{array}
    \right)) \\ &=& 2d(Z,\mathbb C) \\ &=& 2d(u,\mathbb C) \\ &=& 2 \sin\theta.
\end{eqnarray*}

We get $$ 2\sin\theta \geq \gamma(\phi_1,\phi_2)\geq \frac{2\sin\theta}{\sqrt{1+\sin^2\theta}}.$$

It is to be noted $\beta(\phi_1,\phi_2)=\sqrt{2}$ (See Example 3.2 of \cite {BhS} )\footnote
{This computation in \cite{BhS} is erroneous. However the result is clear in view of the fact that
$\langle x_1, x_2\rangle  = 0$ for any common representation $(\mathcal F, x_1, x_2)$ of $(\phi _1, \phi _2)$ due to Lemma \ref{identification}}. Therefore $$\gamma(\phi_1,\phi_2)\neq \beta(\phi_1,\phi_2)\sqrt{4-\beta^2(\phi_1,\phi_2)}.$$

{\bf Acknowledgements :}  {The first author acknowledges support through JC Bose Fellowship project and the
 second  author  thanks DST-Inspire for financial support. The title of the paper is
 inspired by Halmos \cite{Halmos}.}
\bibliographystyle{amsplain}

\end{document}